\documentclass[11pt]{article}
\usepackage{amsfonts,amsmath,amssymb,bbm,amsthm,subcaption,float,multicol,tikz,braids}
\usepackage[margin=1in]{geometry}

\theoremstyle{theorem}
\newtheorem{theorem}{Theorem}[subsection]
\newtheorem{proposition}[theorem]{Proposition}
\newtheorem{lemma}[theorem]{Lemma}
\newtheorem{corollary}[theorem]{Corollary}
\newtheorem{conjecture}{Conjecture}
\theoremstyle{definition}
\newtheorem{definition}[theorem]{Definition}
\newtheorem*{ack}{Acknowledgements}

\newtheorem{note}[theorem]{Note}
\newtheorem{example}[theorem]{Example}
\newtheorem{question}{Question}

\title{Quadratic $d$-numbers}
\date{\today}
\author{Andrew Schopieray}

\begin{document}

\maketitle

\begin{abstract}
Here we constructively classify quadratic $d$-numbers: algebraic integers in quadratic number fields generating Galois-invariant ideals.  We prove the subset thereof maximal among their Galois conjugates in absolute value is discrete in $\mathbb{R}$.  Our classification provides a characterization of those real quadratic fields containing a unit of norm -1 which is known to be equivalent to the existence of solutions to the negative Pell equation.  The notion of a weakly quadratic fusion category is introduced whose Frobenius-Perron dimension necessarily lies in this discrete set.  Factorization, divisibility, and boundedness results are proven for quadratic $d$-numbers allowing a systematic study of weakly quadratic fusion categories which constitute essentially all known examples of fusion categories having no known connection to classical representation theory.
\end{abstract}

\begin{section}{Introduction}
\par The set of $d$-numbers, in the sense of \cite[Definition 1.1]{codegrees}, are those algebraic integers generating Galois-invariant ideals in the ring of all algebraic integers.  Our main result is Theorem \ref{genthm} (together with Example \ref{excomplex}): a complete classification of $d$-numbers in quadratic extensions of $\mathbb{Q}$.  For a given real quadratic extension $\mathbb{Q}(\sqrt{N})$ for some square-free $N\in\mathbb{Z}_{\geq2}$, the classification is naturally divided based on the norm of the fundamental unit, which translates to a characterization of the solvability of the negative Pell equation (Corollary \ref{pell}) dependent on the nonexistence of a $d$-number of square-free norm not dividing $N$.  As a result of Theorem \ref{genthm}, we prove (Corollary \ref{corthree}) that quadratic $d$-numbers have a unique factorization and (Theorem \ref{dis}) the subset of quadratic $d$-numbers $\alpha$ with nontrivial Galois conjugate $\alpha'$ such that $\alpha\geq\alpha'\geq1$ is a discrete subset of $\mathbb{R}$ and we provide a quartic polynomial bound in $M\in\mathbb{Z}_{\geq1}$ for the cardinality of this set contained in the interval $[1,M]$ (Proposition \ref{boond}).

\par The main application of these results apart from number-theoretic interest is the study of fusion categories (over $\mathbb{C}$) which arise naturally from the representation theory of finite groups, quantum groups, and the study of subfactors of von Neumann algebras.  Formal codegrees, including global and Frobenius-Perron dimensions of fusion categories \cite{codegrees}, and higher Gauss sums of premodular categories \cite{highergauss} are examples of $d$-numbers.  In the ongoing search for small-index subfactors, the relative scarcity of $d$-numbers among all algebraic integers has been used to eliminate spurious candidates from classification results \cite{calegari,vines}.  As such, our results have implications for the study of modular tensor categories which are requisite objects of study in rational conformal field theory and vertex operator algebras \cite{conformal,kawaconformal}, quantum computation \cite{ericwang} and invariants of knots and links \cite{turaev}.  The above references are not meant to be comprehensive, but serve as a gateway for the interested reader to the depths of the current literature. 

\par There is a wealth of research which has attempted to characterize families of \emph{weakly integral} fusion categories, i.e.\!\! those whose Frobenius-Perron dimension is a positive integer \cite{paul5,paul3,paul2,dong1,pq,jordan,pointed,natale2014,natale2016}.  The family of weakly integral metaplectic modular tensor categories related to the representation theory of special orthogonal groups has received particular attention \cite{metaplectic,2019arXiv190104462D,2018arXiv180800698G}.  It is conjectured that weakly integral braided fusion categories are unique among braided fusion categories in that their associated braid group representations factor over a finite group \cite{finiteness}.  Such a vast collection of research is possible because, among other reasons, all $d$-numbers in $\mathbb{Q}$ are classified (integers) and have unique factorization (the fundamental theorem of arithmetic).  Outside of $\mathbb{Q}$ the situation is understandably more complicated.  Despite $d$-numbers in real quadratic number fields being dense in $\mathbb{R}$, those which are potentially Frobenius-Perron dimensions of fusion categories form a discrete set (Theorem \ref{dis}).  This allows one to comprehensively list prospective Frobenius-Perron dimensions of fusion categories which lie in a quadratic extension of $\mathbb{Q}$ (see Example \ref{threeplus}).  The major leap in difficulty for such a result is that these $d$-numbers are now being drawn from an infinite collection of fields, which makes Theorem \ref{dis} quite surprising.  Motivated by the existing literature we label fusion categories whose Frobenius-Perron dimension lies in a quadratic extension of $\mathbb{Q}$ as \emph{weakly quadratic} fusion categories.  We initiate their systematic study by proving factorization and divisibility conditions for quadratic $d$-numbers (Corollary \ref{corthree}, Proposition \ref{lemdiv}) and strong lower bounds on the Frobenius-Perron dimension of a weakly quadratic fusion category which lies in a quadratic field whose fundamental unit has norm $-1$ (Proposition \ref{thmbound}).  This provides evidence that there may exist real quadratic number fields that contain no Frobenius-Perron dimensions of fusion categories other than integers.  We finish by showing the non-existence of an infinite family pseudounitary weakly quadratic fusion categories with ``small'' factorizations.  All examples of fusion categories arising from the theory of subfactors automatically satisfy the assumption of pseudounitarity (more so, unitarity) as do generalized near-group categories \cite{thornton2012generalized} which are the largest class of weakly quadratic categories of conjectural existence.

\par Sections \ref{sec:quad} and \ref{dnum} describe all prerequisite knowledge of number theory used in the remainder of the exposition.  Sections \ref{sec:fus} and \ref{sec:dim} provide a basic introduction to fusion categories and different measurements of dimension; these sections are less self-contained than Sections \ref{sec:quad} and \ref{dnum}.  Section \ref{secclass} contains two of the main results of the paper: Theorem \ref{genthm} (classification of quadratic $d$-numbers) and Theorem \ref{dis} (discreteness of Galois maximal quadratic $d$-numbers).  Section \ref{sec:toward} contains corollaries of the classification and example classification results for small quadratic $d$-numbers.  Section \ref{sec:fam} describes the known families of weakly quadratic fusion categories in the context of Theorem \ref{genthm} and finally Section \ref{sec:non} contains a series of lemmas proving the non-existence of pseduounitary fusion categories of particular Frobenius-Perron dimension.  Lastly we finish with a discussion of further directions of research, open questions, and conjectures.
\end{section}

\begin{ack}
We would like to thank Terry Gannon, Pinhas Grossman, Masaki Izumi, Victor Ostrik, and Ana Ros Camacho for their comments during the preparation of this manuscript.
\end{ack}
 
 
\begin{section}{Preliminaries}

All considered fields $\mathbb{K}$ are algebraic over $\mathbb{Q}$, the rational numbers.  We use $[\mathbb{K}:\mathbb{Q}]$ to represent the degree of the field extension $\mathbb{K}/\mathbb{Q}$ (the dimension of $\mathbb{K}$ as a $\mathbb{Q}$-vector space).  For algebraic extensions this coincides with the cardinality of $\mathrm{Gal}(\mathbb{K}/\mathbb{Q})$, the group of $\mathbb{Q}$-invariant field automorphisms of $\mathbb{K}$.  If $\alpha\in\mathbb{K}$ is the root of a monic polynomial with integer coefficients then $\alpha$ is an \emph{algebraic integer}.  The set of algebraic integers in $\mathbb{K}$ has the structure of a ring $\mathcal{O}_\mathbb{K}$ while the total collection of algebraic integers (in $\overline{\mathbb{Q}}$, the algebraic closure of $\mathbb{Q}$) will be denoted $\mathbb{A}$.  Within our context $\|\cdot\|:\mathbb{K}\to\mathbb{Q}$, the field norm of $\mathbb{K}$ over $\mathbb{Q}$, may be computed for any $\alpha\in\mathbb{K}$ by
\begin{equation}
\|\alpha\|=\prod_{\tau\in\mathrm{Gal}(\mathbb{K}/\mathbb{Q})}\tau(\alpha),
\end{equation}
which restricts to a map $\mathcal{O}_\mathbb{K}\to\mathcal{O}_\mathbb{Q}=\mathbb{Z}$.  The superscript $R^\times$ will be used to indicate the set of all invertible elements of a ring $R$.  In particular the elements $\alpha\in\mathcal{O}_\mathbb{K}^\times$ satisfy $\|\alpha\|=\pm1$.


\begin{subsection}{Quadratic number fields}\label{sec:quad}
Algebraic number fields $\mathbb{K}$ with $[\mathbb{K}:\mathbb{Q}]=2$ are in 1:1 correspondence with square-free integers $N\neq0,1$ such that $\mathbb{K}\cong\mathbb{Q}(\sqrt{N})$.  We abbreviate $\mathcal{O}_{\mathbb{Q}(\sqrt{N})}=:\mathcal{O}_N$.  It is well-known that
\begin{equation}
\mathcal{O}_N=\{a+b\omega:a,b\in\mathbb{Z}\}\text{ where }\omega=\left\{\begin{array}{ccl}\sqrt{N} & : & N\equiv2,3\pmod{4} \\ (1/2)(1+\sqrt{N}) & : & N\equiv1\pmod{4}\end{array}\right..
\end{equation}
Unless indicated otherwise, $\sigma$ will be the nontrivial element of $\mathrm{Gal}(\mathbb{Q}(\sqrt{N})/\mathbb{Q})$;  that is to say $\sigma(\sqrt{N})=-\sqrt{N}$.  Thus for all $a,b\in\mathbb{Q}$,
\begin{equation}
\|a+b\sqrt{N}\|=(a+b\sqrt{N})(a-b\sqrt{N})=a^2-Nb^2.
\end{equation}
For square-free $N\in\mathbb{Z}_{<0}$, $\mathcal{O}_N^\times$ is finite.  Specifically, $\mathcal{O}^\times_{-1}=\{\pm1,\pm i\}$, $\mathcal{O}^\times_{-3}=\{1,(1/2)(\pm1\pm\sqrt{-3})\}$, and $\mathcal{O}^\times_{N}=\{\pm1\}$ for $N\neq-1,-3$.  When $N\in\mathbb{Z}_{>0}$, Dirichlet's unit theorem \cite[Theorem 13.1.1]{alaca} implies there exists $\epsilon_N\in\mathcal{O}_N^\times$ which is unique subject to $|\epsilon_N|>1$, such that for all $u\in\mathcal{O}_N^\times$, $u=\pm\epsilon_N^m$ for some $m\in\mathbb{Z}$.  We will refer to $\epsilon_N$ as the \emph{fundamental unit} of $\mathcal{O}_N$.   There exist multiple algorithms to compute $\epsilon_N$ for a fixed square-free $N\in\mathbb{Z}_{\geq2}$ (see \cite{lenstra} for an expository look).  The table provided below will be useful for checking examples for small $N$ where we use the decomposition $\epsilon_N=(1/2)(t_N+u_N\sqrt{N})$ i.e. \!the minimal polynomial of $\epsilon_N$ is $x^2-t_Nx+\|\epsilon_N\|$.
\begin{figure}[H]
\centering
\begin{equation*}
\begin{array}{|c|c|c|c|c|c|c|c|c|c|c|c|c|c|c|c|c|c|}
\hline N & 2 & 3 & 5 & 6 & 7     & 10 & 11 & 13 & 14 & 15 & 17 & 19  & 21   & 22   & 23  & 26 & 29    \\\hline
	 t_N & 2 & 4 & 1 & 10 & 16 &  6  & 20 &  3  & 30 & 8   &  8 & 340 & 5     & 394 & 48  & 10 &  5  \\\hline
	 u_N & 2 & 2& 1 & 4   & 6   &  2  &  6  &  1  &  8  & 2   &  2 &  78  & 1    &  84  & 10  &  2  &  1  \\\hline
\hline N & 30 & 31     &  33  & 34 & 35 & 37 & 38 & 39  & 41 & 42 & 43    & 46    & 47  & 51   & 53 & 55    & 57 \\\hline
	 t_N & 22 & 3040 &  46  & 70 & 12  & 12 & 74& 50  & 64 & 26 & 6964 &48670& 96 & 100  & 7   &  178& 302\\\hline
	u_N  & 4  & 546   &  8    & 12 & 2   &  2  & 12 & 8   & 10 & 4   & 1062 &7176  & 14  & 14   & 1  &   24  &  40\\\hline
\end{array}
\end{equation*}
    \caption{Fundamental units $\epsilon_N=(1/2)(t_N+u_N\sqrt{N})$ for square-free $N\in\mathbb{Z}_{\geq2}$}%
    \label{figunit}%
\end{figure}
\end{subsection}


\begin{subsection}{$d$-numbers}\label{dnum}
The following definition was introduced by Ostrik in \cite{codegrees}.
\begin{definition}
An algebraic integer $\alpha$ is a \emph{$d$-number} if the ideal it generates in $\mathbb{A}$ is invariant under the action of the absolute Galois group $\mathrm{Gal}(\overline{\mathbb{Q}}/\mathbb{Q})$.
\end{definition}
The set of $d$-numbers in an algebraic number field $\mathbb{K}$ will be denoted $\mathfrak{D}_\mathbb{K}$, and if $\mathbb{K}=\mathbb{Q}(\sqrt{N})$ for some square-free integer $N$, $\mathfrak{D}_N$ for brevity.  There are many equivalent definitions of $d$-numbers.
\begin{lemma}[{\cite[Lemma 2.7]{codegrees}}]\label{def}
Let $\alpha\in\mathbb{A}$.  The following are equivalent:
\begin{enumerate}
\item[\textnormal{(i)}] $\alpha$ is a $d$-number;
\item[\textnormal{(ii)}] For any $\tau\in\mathrm{Gal}(\overline{\mathbb{Q}}/\mathbb{Q})$, $\alpha/\tau(\alpha)\in\mathbb{A}$;
\item[\textnormal{(iii)}] For any $\tau\in\mathrm{Gal}(\overline{\mathbb{Q}}/\mathbb{Q})$, $\alpha/\tau(\alpha)\in\mathbb{A}^\times$;
\item[\textnormal{(iv)}] There exists a positive integer $m$ such that $\alpha^m\in\mathbb{Z}\cdot\mathbb{A}^\times$;
\item[\textnormal{(v)}] Let $p(x)=x^n+a_1x^{n-1}+\cdots+a_n$ be the minimal polynomial of $\alpha$ over $\mathbb{Q}$ (so, $a_i\in\mathbb{Z}$).  Then for any $i=1,\ldots,n$ the number $(a_i)^n$ is divisible by $(a_n)^i$;
\item[\textnormal{(vi)}] There exists a polynomial $p(x)=x^m+a_1x^{m-1}+\cdots+a_m\in\mathbb{Z}[x]$ such that $p(\alpha)=0$ and for any $i=1,\ldots,m$ the number $(a_i)^m$ is divisible by $(a_m)^i$.
\end{enumerate}
\end{lemma}
Some examples follow immediately from the above equivalences.  Part (iv) with $m=1$ implies all integer multiples of units are $d$-numbers.  One can use (vi) to discover $d$-numbers not of this type.  For example, let $m\in\mathbb{Z}_{\geq1}$ and $n\in\mathbb{Z}$, and define
\begin{equation}
p_{n,m}(x):=x^m+nx^{m-1}+\cdots nx+n.
\end{equation}
All roots of all $p_{n,m}$ are $d$-numbers, although this sporadic method of construction is not very useful for classification purposes.  
The following definition and lemma are insinuated and used in the proof of \cite[Lemma 2.7]{codegrees}.
\begin{definition}\label{deforder}
If $\alpha\in\mathfrak{D}_\mathbb{K}$, we say the least $k\in\mathbb{Z}_{\geq1}$ such that $\alpha^k=nu$ for some $n\in\mathbb{Z}$ and $u\in\mathcal{O}_\mathbb{K}^\times$ is $\mathrm{ord}(\alpha)$, the \emph{order} of $\alpha$ (Lemma \ref{def} (iv) implies existence).
\end{definition}
Usage of the term \emph{order} is appropriate by offering a slightly more generalized, but not necessarily more useful, definition of $d$-numbers (refer also to Note \ref{ostriknote}).  Casually, we use the term \emph{generalized $d$-number} to describe any algebraic \emph{number} $\alpha$ such that there exists $k\in\mathbb{Z}_{\geq1}$ with $\alpha^k\in\mathbb{Q}\cdot\mathcal{O}_\mathbb{K}^\times$.  Generalized $d$-numbers form an abelian group under multiplication, unlike $d$-numbers which are not closed under inversion.  The quotient by the subgroup $\mathbb{Q}\cdot\mathcal{O}_\mathbb{K}^\times$ will be denoted $\mathfrak{G}_\mathbb{K}$.

\begin{lemma}[{cf. \!\cite[Lemma 2.7 proof]{codegrees}}]\label{lemmagal}
Let $\mathbb{K}$ be an algebraic number field.  If $\alpha\in\mathfrak{G}_\mathbb{K}$, then $\mathrm{ord}(\alpha)$ divides $[\mathbb{K}:\mathbb{Q}]$.
\end{lemma}

\begin{proof}
If $\alpha_1,\ldots,\alpha_\ell$ are the Galois conjugates of $\alpha$, Lemma \ref{def} (iii) implies $\alpha/\alpha_i\in\mathcal{O}_\mathbb{K}^\times$ for $i=1,\ldots,\ell$, hence $\alpha^\ell=\|\alpha\|u$ for some $u\in\mathcal{O}_\mathbb{K}^\times$ and $\ell$ divides $[\mathbb{K}:\mathbb{Q}]$.  Elementary group theory then implies $\mathrm{ord}(\alpha)$ divides $[\mathbb{K}:\mathbb{Q}]$ as well.
\end{proof}

The sets $\mathfrak{D}_N$ are amenable to study and classification because one can write explicit criteria for inclusion using Lemma \ref{def} (iii) and (v) and Lemma \ref{lemmagal}.
\begin{lemma}\label{lemma}
Let $N\in\mathbb{Z}_{\geq2}$ be square-free and $\alpha=a+b\sqrt{N}\in\mathcal{O}_N$ for some $a,b\in\frac{1}{2}\mathbb{Z}$. The following are equivalent:
\begin{enumerate}
\item[(i)] $\alpha\in\mathfrak{D}_N$;
\item[(ii)] $(a^2+Nb^2+2ab\sqrt{N})/(a^2-Nb^2)\in\mathcal{O}_N^\times$;
\item[(iii)] $\mathrm{ord}(\alpha)=1,2$;
\item[(iv)] $4a^2/(a^2-Nb^2)\in\mathbb{Z}$.
\end{enumerate}
\end{lemma}

The set $\mathfrak{D}_\mathbb{K}$ has the structure of a monoid, but not a ring, with an exception of $\mathfrak{D}_\mathbb{Q}=\mathbb{Z}$.

\begin{example}
For any square-free $N\in\mathbb{Z}_{\geq2}$, $\mathfrak{D}_N$ is not closed under addition.  To see this, Lemma \ref{lemma} (iv) implies $\alpha:=a+b\sqrt{N}\not\in\mathfrak{D}_N$ if $4a^2/(a^2-Nb^2)\not\in\mathbb{Z}$.  Hence $b\sqrt{N}\in\mathfrak{D}_N$ for all $b\in\mathbb{Z}$ and $N\in\mathfrak{D}_N$.  We claim that $N+b\sqrt{N}$ is not a $d$-number for many choices of $b\in\mathbb{Z}$ since for fixed square-free $N\in\mathbb{Z}_{\geq2}$, any choice of $b\not\equiv N\pmod{2}$ implies $4N^2/(N^2-Nb^2)\not\in\mathbb{Z}$ and hence $N+b\sqrt{N}\not\in\mathfrak{D}_N$.
\end{example}

\begin{example}[Complex quadratic $d$-numbers]\label{excomplex}
Assume $N$ is a square-free negative integer different from $-1$ and $-3$, so that $\mathcal{O}_N^\times=\{\pm1\}$ (see Section \ref{sec:quad}).  Lemma \ref{lemma} (iii) states $\alpha\in\mathfrak{D}_N$ if and only if $\mathrm{ord}(\alpha)=1,2$.  In the latter case, $\alpha^2\in N\cdot\mathbb{Z}$ since $\mathbb{Q}(\alpha)\subset\mathbb{Q}(\sqrt{N})$.  Moreover
\begin{equation}
\mathfrak{D}_N=\{\ell(\sqrt{N})^\delta:\ell\in\mathbb{Z},\delta=0,1\}\qquad\text{ for square-free }N\in\mathbb{Z}_{<0},N\neq-1,-3.
\end{equation}
\par For $N=-1,-3$ we will use Lemma \ref{lemma} (ii), which implies $\alpha=a+bi\in\mathfrak{D}_{-1}$ for $a,b\in\mathbb{Z}$ if and only if one of four conditions are true, corresponding to the four units of $\mathcal{O}_{-1}$.  Two of the conditions are $a^2-b^2+2abi=\pm(a^2+b^2)$ hence $b(ai-b)=0$ or $a(bi+a)=0$.  This implies $\alpha\in\mathbb{Z}$ in the former case, and $\alpha\in i\cdot\mathbb{Z}$ in the latter.  The other two conditions are $a^2-b^2+2abi=\pm i(a^2+b^2)$, hence $(a-b)(a+b)=0$ or $(a-b)(a+b)=i(a+b)^2$.  The latter implies $\alpha\in(-1+i)\cdot\mathbb{Z}$ and the former implies $\alpha\in(1+i)\cdot\mathbb{Z}$.  Moreover
\begin{equation}
\mathfrak{D}_{-1}=\{\ell i^m(1\pm i)^\delta:\ell,m\in\mathbb{Z},\delta=0,1\}.
\end{equation}
\par Similarly $\alpha=a+b\sqrt{-3}\in\mathfrak{D}_{-3}$ for some $a,b\in\frac{1}{2}\mathbb{Z}$ if and only if one of Equations (\ref{ferst})--(\ref{laast}) is true; the first two correspond to the units $\pm1$, and the latter two corresponding to the units $\pm(1/2)(1\pm\sqrt{-3})$ and $\mp(1/2)(1\pm\sqrt{-3})$.  We must have
\begin{align}
b(3b+a\sqrt{-3})&=0, \label{ferst}\\
a(a+b\sqrt{-3})&=0, \label{mid1} \\
3(a+b)(a-b)&=\mp(a\pm b)(a\pm 3b)\sqrt{-3},\text{ or} \label{mid2}\\
(a-3b)(a+3b)&=\mp(a\pm b)(a\pm 3b)\sqrt{-3}\label{laast}.
\end{align}
Equations (\ref{ferst}) and (\ref{mid1}) imply $\alpha=\ell(\sqrt{-3})^\delta$ for some $\ell\in\mathbb{Z}$ and $\delta=0,1$.  Equation (\ref{mid2}) implies $a=\pm b$ in which case $\alpha\in(1/2)(1\pm\sqrt{-3})\cdot\mathbb{Z}$, while Equation (\ref{laast}) implies $a=\pm3b$ in which case $\alpha\in(1/2)(3\pm\sqrt{-3})\cdot\mathbb{Z}$.  These are not minimal generators as $\pm(1/2)(1\mp\sqrt{-3})\sqrt{-3}=(1/2)(3\pm\sqrt{-3})$.  Moreover
\begin{equation}
\mathfrak{D}_{-3}=\{\ell((1/2)(1
\pm\sqrt{-3}))^m(\sqrt{-3})^\delta:\ell,m\in\mathbb{Z},\delta=0,1\}.
\end{equation}
\end{example}
\end{subsection}


\begin{subsection}{Fusion categories}\label{sec:fus}
A fusion category (over $\mathbb{C}$) is a $\mathbb{C}$-linear semisimple rigid monoidal ($\otimes$) category with finitely-many simple objects and simple monoidal unit $\mathbbm{1}$.  Standard references for the subject are \cite[Chapter 9]{tcat} and \cite{ENO}.  For the purposes of this paper, the most important characteristic of fusion categories is their semisimple \emph{fusion rules}.  That is to say if $\{X_i\}_{i\in I}$ index the isomorphism classes of simple objects of a fusion category $\mathcal{C}$, $X_i\otimes X_j\cong\oplus_{i\in I}N_{ij}^kX_k$ for some \emph{fusion coefficients} $N_{ij}^k\in \mathbb{Z}_{\geq0}$.  This phenomenon can be studied on the level of \emph{fusion rings} by considering the Grothendieck rings of fusion categories, but the number-theoretic peculiarities we will be describing are only guaranteed when a fusion ring arises in this fashion (such a fusion ring is said to be \emph{categorifiable}). For non-experts, it is likely that the only unfamiliar adjective associated to fusion categories is ``rigid''.  An object $X$ in a monoidal category $\mathcal{C}$ is \emph{rigid} if $X$ has left and right \emph{dual} objects $X^\ast$ and $^\ast X$, respectively, and monoidal category $\mathcal{C}$ is rigid if all $X\in\mathcal{C}$ are rigid \cite[Section 2.10]{tcat}.  Dual objects are unique up to isomorphism, restricted by the fact that the functor $X^\ast\otimes-$ is left adjoint to $X\otimes-$, and $-\otimes\,^\ast X$ is right adjoint to $-\otimes X$.  The assumption of semisimplicity for a fusion category $\mathcal{C}$ ensures that $^\ast X\cong X^\ast$ for all  $X\in\mathcal{C}$ \cite[Proposition 4.8.1]{tcat} so we will choose to always write \emph{the} dual of any object $X$ in a fusion category $\mathcal{C}$ as $X^\ast$.  This ambidexterity of duality does not hold true in non-semisimple examples \cite[Example 7.19.5]{tcat}.  Morphisms in fusion categories are often illustrated by diagrams up to local isotopy subject to a graphical calculus \cite[Section 2.3]{BaKi}, where the above adjunctions for a given object $X$ give rise to ``cup'' ($\cup:\mathbbm{1}\to X\otimes X^\ast$) and ``cap'' ($\cap:X^\ast\otimes X\to\mathbbm{1}$) diagrams (morphisms) which satisfy the following compatibility in $\mathrm{End}(X)$:
\begin{equation}
\begin{tikzpicture}[scale=0.3]
\draw (-1,2) -- (-1,4);
\draw (-1,2) arc (180:360:0.5);
\draw (1,2) arc (0:180:0.5);
\draw (1,2) -- (1,0);
\end{tikzpicture}
\qquad
\begin{tikzpicture}[scale=0.3]
\node at (0,0) {\null};
\node at (0,2) {$=$};
\end{tikzpicture}
\qquad
\begin{tikzpicture}[scale=0.3]
\draw (0,0) -- (0,4);
\end{tikzpicture}
\qquad
\begin{tikzpicture}[scale=0.3]
\node at (0,0) {\null};
\node at (0,2) {$=$};
\end{tikzpicture}
\qquad
\begin{tikzpicture}[scale=0.3]
\draw (1,2) -- (1,4);
\draw (1,2) -- (1,4);
\draw (0,2) arc (180:360:0.5);
\draw (0,2) arc (0:180:0.5);
\draw (-1,2) -- (-1,0);
\end{tikzpicture}
.
\end{equation}
Vertical lines represent the identity morphism and morphisms are composed from bottom-to-top.  We will use these diagrams as short-hand in definitions but will not use them as a method of reasoning in proofs.

\begin{example}
Let $G$ be a finite group.  The category of complex $G$-graded vector spaces $\mathrm{Vec}_G$ is a fusion category whose simple objects are isomorphic to the 1-dimensional vector spaces indexed by elements of $g$, and fusion rules are given by group multiplication.  One can also equip a non-trivial associative isomorphism $\omega:G\times G\times G\to\mathbb{C}^\times$ to the group multiplication to produce the \emph{twisted} examples $\mathrm{Vec}^\omega_G$ \cite[Example 2.3.8]{tcat}.  Finite-dimensional complex representations of $\mathbb{C}G$ (or any finite-dimensional complex Hopf algebra) provide another family of examples \cite[Chapter 5]{tcat}.
\end{example}

Commutativity is realized in a fusion category $\mathcal{C}$ by the notion of a \emph{braiding}: a collection of natural isomorphisms $\sigma_{X,Y}:X\otimes Y\to Y\otimes X$ for all simple $X,Y\in\mathcal{C}$ satisfying compatibility constraints \cite[Definition 8.1.1]{tcat}.  These compatibilities ensure $\sigma_{-,-}$ (illustrated as a crossing $\begin{tikzpicture}[scale=0.22]
\braid a_1;
\end{tikzpicture}$ in the graphical calculus) satisfies the \emph{Yang-Baxter equation}
\begin{equation}
\begin{tikzpicture}[scale=0.4]
\braid a_1 a_2 a_1;
\end{tikzpicture}
\qquad
\begin{tikzpicture}[scale=0.4]
\node at (0,0) {\null};
\node at (0,1.5) {$=$};
\end{tikzpicture}
\qquad
\begin{tikzpicture}[scale=0.4]
\braid a_2 a_1 a_2;
\end{tikzpicture}
\end{equation}
as isomorphims $X\otimes Y\otimes X\to Z\otimes Y\otimes X$ for all simple $X,Y,Z\in\mathcal{C}$.  The two extremes of braided fusion categories are \emph{symmetric} braidings, for which $\sigma_{Y,X}\sigma_{X,Y}=\mathrm{id}_{X,Y}$ for all simple $X,Y\in\mathcal{C}$, and \emph{nondegenerate} braidings for which the \emph{symmetric center}, whose isomorphism classes of objects form the set
\begin{equation}
\mathcal{C}':=\{X\in\mathcal{C}:\sigma_{Y,X}\sigma_{X,Y}=\mathrm{id}_{X,Y}\text{ for all }Y\in\mathcal{C}\},
\end{equation}
is trivial, i.e. \!equivalent to $\mathrm{Vec}$, the category of complex finite-dimensional vector spaces (the unit object is always in the symmetric center).  All symmetric fusion categories ($\mathcal{C}=\mathcal{C}'$) are \emph{Tannakian} (braided equivalent to $\mathrm{Rep}(G)$ for a finite group $G$), or \emph{super-Tannakian} (Tannakian, with the braiding twisted by a central element of the finite group of order 2) by a result of Deligne \cite{deligne1,deligne2}.
\end{subsection}


\begin{subsection}{Notions of dimension in a fusion category}\label{sec:dim}

Denote the set of isomorphism classes of simple objects of a fusion category $\mathcal{C}$ by $\mathcal{O}(\mathcal{C})$ for brevity.  Every $X\in\mathcal{O}(\mathcal{C})$ is isomorphic to its double-dual $X^{\ast\ast}:=(X^\ast)^\ast$ \cite[Proposition 4.8.1]{tcat} via an isomorphism $a_X:X\to X^{\ast\ast}$.  One can then define, in the spirit of \cite[Section 2.2]{mug2}, the \emph{squared norm} of $X$ as
\begin{equation}
|X|^2:=\mathrm{Tr}_X(a_X)\mathrm{Tr}_{X^\ast}((a_X^{-1})^\ast),
\end{equation} 
where for any $f:X\to X^{\ast\ast}$, the (left) \emph{categorical trace} of $f$, $\mathrm{Tr}_X(f)\in\mathrm{End}(\mathbbm{1})$, is defined diagrammatically by
\begin{equation}
\begin{tikzpicture}[scale=0.5]
\draw (1,2) arc (0:180:1);
\draw (-1,1) arc (180:360:1);
\draw (-1,1) -- node[circle,pos=0.5,fill=white,draw=black,inner sep=0.5mm] {$f$} (-1,2);
\draw (1,1) -- (1,2);
\end{tikzpicture}
.
\end{equation}
The squared norm is independent of the choice of $a_X$ because $a_X$ itself is unique up to a scalar.  Therefore, one can assign a natural sense of size to each fusion category, the \emph{global dimension}
\begin{equation}
\dim(\mathcal{C}):=\sum_{X\in\mathcal{O}(\mathcal{C})}|X|^2.
\end{equation}
As shown in the proof of \cite[Theorem 2.3]{ENO}, $\dim(\mathcal{C})-1$ (and hence $\dim(\mathcal{C})$) is a totally positive algebraic integer.  A set of isomorphisms $a:=\{a_X:X\to X^{\ast\ast}\}_{X\in\mathcal{O}(\mathcal{C})}$ such that $a_{X\otimes Y}=a_X\otimes a_Y$ for all $X,Y\in\mathcal{O}(\mathcal{C})$ is known as a \emph{pivotal structure} and together with this additional structure $\mathcal{C}$ is a \emph{pivotal fusion category}.  We then define $\dim_a(X):=\mathrm{Tr}_X(a_X)$ for all $X\in\mathcal{O}(\mathcal{C})$ so that $\dim_a$ is a character of the Grothendieck ring of $\mathcal{C}$ and $|X|^2=\dim_a(X)\dim_a(X^\ast)=|\dim(X)|^2$ \cite[Proposition 2.9]{ENO}.  All known examples of fusion categories possess pivotal structures $a$ which naturally make $\dim_a(X)$ a square root of $|X|^2$.
\begin{definition}
Let $\mathcal{C}$ be a pivotal fusion category.  A pivotal structure $a$ on $\mathcal{C}$ is \emph{spherical} if $\dim_a(X)=\dim_a(X^\ast)$ for all $X\in\mathcal{O}(\mathcal{C})$.  In this case we refer to $\mathcal{C}$ as a \emph{spherical fusion category}.
\end{definition}
Thus $|X|^2=\dim_a(X)^2$ for all simple $X$ in a spherical fusion category $\mathcal{C}$.  The existence (or nonexistence) of a canonical spherical structure for any fusion category is a fundamental open question.
\begin{definition}
A \emph{modular tensor category} is a non-degenerately braided spherical fusion category.
\end{definition}
Modular tensor categories are ubiquitous as the \emph{Drinfeld center} $\mathcal{Z}(\mathcal{C})$ \cite[Section 8.5]{tcat} of any spherical fusion category $\mathcal{C}$ is a modular tensor category.  The dimensions of objects in modular tensor categories are strictly limited by their global dimension.
\begin{lemma}[Proposition 8.14.6 \cite{tcat}]\label{modular}
If $\mathcal{C}$ is a modular tensor category with pivotal structure $a$, then $\dim(\mathcal{C})/\dim_a(X)^2$ is an algebraic integer for all $X\in\mathcal{O}(\mathcal{C})$.
\end{lemma}
There is another notion of dimension arising from the \emph{Frobenius-Perron theory} of fusion categories.  To this end, let $\mathcal{C}$ be a rank $n$ fusion category with simple objects $\mathbbm{1}=X_0,X_1,\ldots,X_{n-1}$.  For a fixed $0\leq i\leq n-1$, consider the \emph{fusion matrix} of nonnegative integers $M_i:=(N_{ij}^k)_{0\leq j,k\leq n-1}$.  Each $M_i$ has a maximal positive real eigenvalue by the Frobenius-Perron theorem \cite[Theorem 3.2.1 (1)]{tcat} which we will denote by $\mathrm{FPdim}(X_i)$ and extend linearly to any $X\in\mathcal{C}$.  Then, the Frobenius-Perron dimension of a fusion category can be defined as
\begin{equation}
\mathrm{FPdim}(\mathcal{C})=\sum_{X\in\mathcal{O}(\mathcal{C})}\mathrm{FPdim}(X)^2.
\end{equation}
\begin{lemma}[Proposition 8.15 \cite{ENO}]\label{subcategory}
Let $\mathcal{C}$ be a fusion category.  If $\mathcal{D}\subset\mathcal{C}$ is a fusion subcategory, then $\mathrm{FPdim}(\mathcal{C})/\mathrm{FPdim}(\mathcal{D})$ is an algebraic integer.
\end{lemma}
A fusion category $\mathcal{C}$ is \emph{pseudounitary} if $\mathrm{FPdim}(\mathcal{C})=\dim(\mathcal{C})$.  Each pseudounitary fusion category $\mathcal{C}$ has a unique pivotal structure $a$ such that $\dim_a(X)=\mathrm{FPdim}(X)$ for all $X\in\mathcal{O}(\mathcal{C})$ \cite[Proposition 8.23]{ENO}.  This pivotal structure is spherical and it will be assumed in this exposition that any pseudounitary fusion category is equipped with this distinct pivotal structure.  Take note that pseudounitary modular tensor categories $\mathcal{C}$ enjoy the same divisibility properties for Frobenius-Perron dimensions as in Lemma \ref{modular}: $\mathrm{FPdim}(\mathcal{C})/\mathrm{FPdim}(X)^2$ is an algebraic integer for all $X\in\mathcal{O}(\mathcal{C})$.  Similarly, if $\mathcal{C}$ is a nontrivial fusion category, $\mathrm{FPdim}(\mathcal{C})-1$ is totally positive.  For future reference we combine this with the maximality of $\mathrm{FPdim}(\mathcal{C})$ in its Galois orbit.

\begin{lemma}\label{totally}
Let $\mathcal{C}$ be a fusion category.  If $\tau\in\mathrm{Gal}(\overline{\mathbb{Q}}/\mathbb{Q})$, then $\mathrm{FPdim}(\mathcal{C})\geq\tau(\mathrm{FPdim}(\mathcal{C}))\geq1$.
\end{lemma}

\begin{definition}\label{deffour}
A fusion category $\mathcal{C}$ is \emph{weakly integral} if $\mathrm{FPdim}(\mathcal{C})\in\mathcal{O}_\mathbb{Q}=\mathbb{Z}$ and \emph{weakly quadratic} if $\mathrm{FPdim}(\mathcal{C})\in\mathcal{O}_\mathbb{K}$ where $\mathbb{K}=\mathbb{Q}(\sqrt{N})$ for some positive square-free integer $N$.  The category $\mathcal{C}$ is \emph{strictly quadratic} if it is weakly quadratic but \emph{not} weakly integral.
\end{definition}
\begin{note}
It is important that the concepts of weakly integral and weakly quadratic be defined in terms of Frobenius-Perron dimension and not global dimension (cf. \!\cite[Remark 8.26]{ENO}).  To see why, take any fusion category $\mathcal{C}$ such that $\dim(\mathcal{C})\not\in\mathbb{Z}$ and let $\{\mathcal{C}_i\}_{i\in I}$ be the collection of all Galois conjugates of $\mathcal{C}$ (obtained by applying a Galois automorphism to all structural constants of $\mathcal{C}$).  The global dimensions of $\mathcal{C}_i$ are all the Galois conjugates of $\dim(\mathcal{C})$, hence $\dim\left(\boxtimes_{i\in I}\mathcal{C}_i\right)\in\mathbb{Z}$.  Therefore if weakly integral or weakly quadratic were defined in terms of global dimension, these concepts would not have the desirable trait of being preserved under fusion subcategories, for example.
\end{note}
The following lemma is the primary motivation for the study of $d$-numbers.  It is Theorem 1.8, Corollary 1.3, and Corollary 1.4 of \cite{codegrees} combined so they may be referred to easily in Section \ref{sec:app}.
\begin{lemma}\label{dfp} The global dimension and Frobenius-Perron dimension of a fusion category are $d$-numbers.  If $\mathcal{C}$ is a braided fusion category, then $|X|^2$ is a $d$-number for all $X\in\mathcal{O}(\mathcal{C})$, and if $\mathcal{C}$ has a pivotal structure $a$, then $\dim_a(X)$ is a $d$-number for all $X\in\mathcal{O}(\mathcal{C})$ as well.
\end{lemma}
\begin{proof}
The third conclusion follows from \cite[Corollary 2.9]{codegrees} as $|X|^2=|\dim_a(X)|^2$ for all simple $X$ in a pivotal fusion category.
\end{proof}
This result cannot be passed to dimensions of simple objects in generality.  There exists a fusion category $\mathcal{H}_1$, categorically Morita equivalent to possibly the most-famed fusion category, constructed as the even sectors of the Haagerup subfactor \cite{asaeda1999exotic} (see also Section \ref{fam}).   The category $\mathcal{H}_1$ has a simple object of Frobenius-Perron dimension $(1/2)(1+\sqrt{13})$ \cite[Section 2.20]{pinhasnoah} which is not a $d$-number.  But $\mathcal{H}_1$ is pseudounitary, hence its Drinfeld center is a pseudounitary modular tensor category and thus $\mathrm{FPdim}(X)$ is a $d$-number for all $X\in\mathcal{O}(\mathcal{Z}(\mathcal{H}_1))$.
\end{subsection}
\end{section}


\begin{section}{Classification of quadratic $d$-numbers}\label{secclass}

The sets $\mathfrak{D}_N$ were computed for square-free $N\in\mathbb{Z}_{<0}$ in Example \ref{excomplex} which were attainable because the unit groups $\mathcal{O}_N^\times$ are finite.  When $N>0$, the unit groups are infinite and more machinery is required.  As we have mentioned, integer multiples of units are always $d$-numbers so we will denote the sets of $d$-numbers $\mathfrak{D}_N$ in the following abbreviated manner.
\begin{definition}
If $N\in\mathbb{Z}_{\geq2}$ is square-free, we denote
\begin{equation}
\mathfrak{D}_N=:\langle \beta_0,\beta_1,\ldots,\beta_k\rangle
\end{equation}
for some $\beta_0,\beta_1,\ldots,\beta_k\in\mathfrak{D}_N$, if for all $\alpha\in\mathfrak{D}_N$, there exist $\ell_\alpha,m_\alpha,n_{\alpha,0},n_{\alpha,1},\ldots,n_{\alpha,k}\in\mathbb{Z}$ such that
\begin{equation}
\alpha=\ell_\alpha\epsilon_N^{m_\alpha}\beta_0^{n_{\alpha,0}}\beta_1^{n_{\alpha,1}}\cdots\beta_k^{n_{\alpha,k}}.
\end{equation}
\end{definition}
In particular, Theorem \ref{genthm} describes the monoidal generators $\beta_0,\beta_1,\ldots,\beta_k$ in terms of the fundamental unit $\epsilon_N\in\mathcal{O}_N^\times$ and the reducibility of an associated biquadratic polynomial.

\begin{subsection}{Main theorem and proof}
We will need the following technical lemma and its notation to state our main result.
\begin{lemma}\label{above}
Let $N\in\mathbb{Z}_{\geq2}$ be square-free.  There exist square-free $\kappa_1\neq\kappa_2\in\mathbb{Z}_{\geq2}$ such that the biquadratic $x^4-\kappa_it_Nx^2+\|\epsilon_N\|\kappa_i^2$ is reducible for $i=1,2$, if and only if $\|\epsilon_N\|=1$ in which case $\kappa_1,\kappa_2$ are the unique square-free integers such that $\kappa_1(t_N+2)$ and $\kappa_2(t_N-2)$ are perfect squares.
\end{lemma}
\begin{proof}
\par Let $k\in\mathbb{Z}_{\geq2}$ be square-free and denote $g(x,N,k):=x^4-kt_Nx^2+\|\epsilon_N\|k^2$.  The conditions upon which $g$ splits were given by Driver, Leonard, and Williams in \cite{driver}.  We include the exact statements here so the reader may easily verify the arguments to follow.
\begin{corollary}\label{corollary1}
Let $r,s\in\mathbb{Z}$ such that $r^2-4s=t^2$ for some $t\in\mathbb{Z}$.  Then $f(x)=x^4+rx+s$ is reducible in $\mathbb{Z}[x]$ and
\begin{equation}
f(x)=(x^2+(1/2)(r-t))(x^2+(1/2)(r+t)).
\end{equation}
\end{corollary}
\begin{corollary}\label{corollary2}
Let $r,s\in\mathbb{Z}$ such that $r^2-4s$ is not a perfect square.  Then $f(x)=x^4+rx+s$ is reducible in $\mathbb{Z}[x]$ if and only if there exists $c\in\mathbb{Z}$ such that $c^2=s$ and $2c-r=a^2$ for some $a\in\mathbb{Z}$, in which case
\begin{equation}
f(x)=(x^2+ax+c)(x^2-ax+c).
\end{equation}
\end{corollary}
For $g(x,N,k)$, $r=-kt_N$, $s=\|\epsilon_N\|k^2$, and cases are dependent on whether $r^2-4s=k^2(t_N^2-4\|\epsilon_N\|)$ is a perfect square, hence when $t_N^2-4\|\epsilon_N\|$ is a perfect square.  If $\|\epsilon_N\|=1$ then $r^2-4s=(t_N+2)(t_N-2)$ which is a perfect square if and only if $t_N=2$ which implies $g(x,N,k)=(x^2+k)^2$.  As $k$ was assumed positive, these infinite possible values of $k$ are spurious.  Therefore Corollary \ref{corollary2} is the applicable result when $\|\epsilon_N\|=1$.  There are two possible values of $k$ based on the fact that $k^2=(-k)^2$ and the note following Corollary \ref{corollary2} in \cite{driver} implies their uniqueness for fixed $g$.  Their defining condition is given in the statement of the lemma.  Lastly, assume $\|\epsilon_N\|=-1$ in which case $-k^2$ is never a perfect square (in $\mathbb{R}$), so Corollary \ref{corollary1} is the applicable result.  But the hypotheses require $t_N^2+4$ to be a perfect square, which never occurs when $t_N>0$.  Indeed, if $\ell\in\mathbb{Z}_{\geq1}$, the difference of adjacent squares is $(\ell+1)^2-\ell^2=\ell(\ell+2)$ which is evidently greater than 4 when $\ell\geq2$, and 3 when $\ell=1$.  This also implies $\kappa_1\neq\kappa_2$ in the case $\|\epsilon_N\|=1$ or else the difference of squares $k(t_N+2)-k(t_N-2)=4$.
\end{proof}
\begin{note}For convenience, Figure \ref{figkappa} lists $\kappa_1,\kappa_2$ for small square-free $N\in\mathbb{Z}_{\geq2}$.\end{note}
\begin{theorem}\label{genthm}
Let $N\in\mathbb{Z}_{\geq2}$ be square-free and if $\|\epsilon_N\|=1$, let $\kappa_1,\kappa_2$ be as in Lemma \ref{above}. Then
\begin{equation}
\mathfrak{D}_N=\left\{\begin{array}{lcl} \langle\sqrt{N}\rangle & : & \|\epsilon_N\|=-1 \\
\langle\sqrt{\kappa_1\epsilon_N},\sqrt{\kappa_2\epsilon_N}\rangle & : & \|\epsilon_N\|=1\textnormal{ and }\kappa_1\kappa_2=N \\
\langle\sqrt{N},\sqrt{\kappa_1\epsilon_N}\rangle & : &  \|\epsilon_N\|=1\textnormal{ and }N\kappa_1=\kappa_2 \\
\langle\sqrt{N},\sqrt{\kappa_2\epsilon_N}\rangle & : &  \|\epsilon_N\|=1\textnormal{ and }N\kappa_2=\kappa_1 \\
\langle\sqrt{N},\sqrt{\kappa_1\epsilon_N},\sqrt{\kappa_2\epsilon_N}\rangle & : & \textnormal{else},
\end{array}\right. \label{twentyfour}
\end{equation}
and these generating sets are minimal.
\end{theorem}
\begin{proof}
If $\alpha\in\mathfrak{D}_N$, Lemma \ref{lemma} (iii) implies $\alpha=\ell\epsilon_N^m$ (in which case there is nothing to show) or $\alpha=\sqrt{\ell\epsilon_N^m}$ for some $\ell,m\in\mathbb{Z}$.  If $m=2k$ for $k\in\mathbb{Z}$, then $\sqrt{\epsilon_N^m}\in\mathbb{Q}(\sqrt{N})$ and thus $\sqrt{\ell}\in\mathbb{Q}(\sqrt{N})$ as well.  Hence $\sqrt{\ell\epsilon_N^n}=\ell'\epsilon_N^k\sqrt{N}$ for some $\ell'\in\mathbb{Z}$ and thus $\sqrt{N}$ is a generator of $\mathfrak{D}_N$.  If $m=2k+1$ for some $k\in\mathbb{Z}$ is odd, then $\sqrt{\ell\epsilon_N^n}=\epsilon_N^k\sqrt{\ell\epsilon_N}$.  It suffices to assume $\ell$ is square-free, and we have reduced the problem to understanding when there exists $\ell\in\mathbb{Z}_{\geq2}$ such that $\sqrt{\ell\epsilon}\in\mathbb{Q}(\sqrt{N})$.  There exists such an $\ell$ when conditions of Lemma \ref{above} are satisfied as $\sqrt{\ell\epsilon_N}$ is a root of $x^4-\ell t_Nx^2+\|\epsilon_N\|\ell^2$.  Moreover, if $\|\epsilon_N\|=-1$ there is a unique generator of $\mathfrak{D}_N$ but in the case $\|\epsilon_N\|=1$, there are three.  It is possible there are nontrivial relations between these generators, which we identify below.
\begin{lemma}\label{xyz}
Let $N\in\mathbb{Z}_{\geq2}$ be square-free with $\|\epsilon_N\|=1$ and $\kappa_1,\kappa_2$ be as in Lemma \ref{above}.  If $x,y$ are the least of $\{\kappa_1,\kappa_2,N\}$, and $z$ is the greatest, then $\sqrt{\kappa_1\epsilon_N}$, $\sqrt{\kappa_2\epsilon_N}$, and $\sqrt{N}$ are monoidally dependent in $\mathfrak{D}_N$ if and only if $xy=z$.
\end{lemma}
\begin{proof}
First note that $\sqrt{\kappa_1}\neq\sqrt{\kappa_2}\in\mathbb{Q}(\sqrt{\epsilon_N})$, a quartic extension of $\mathbb{Q}$, as $\mathbb{Q}(\epsilon_N)=\mathbb{Q}(\sqrt{N})$.  Therefore $\mathrm{Gal}(\mathbb{Q}(\sqrt{\epsilon_N})/\mathbb{Q})\cong\mathbb{Z}/2\mathbb{Z}\times\mathbb{Z}/2\mathbb{Z}$ and $\mathbb{Q}(\sqrt{\epsilon_N})$ has three distinct quadratic subfields $\mathbb{Q}(\sqrt{\kappa_1})$, $\mathbb{Q}(\sqrt{\kappa_2})$, and $\mathbb{Q}(\sqrt{\kappa_1\kappa_2})$.  But $\sqrt{\kappa_1\kappa_2}=\epsilon_N^{-1}\sqrt{\kappa_1\epsilon_N}\sqrt{\kappa_2\epsilon_N}\in\mathbb{Q}(\sqrt{N})$, so $\mathbb{Q}(\sqrt{\kappa_1\kappa_2})=\mathbb{Q}(\sqrt{N})$ and we conclude $N,\kappa_1,\kappa_2$ are all distinct.  As $N$ is square-free, $\kappa_1\kappa_2=\gcd(\kappa_1,\kappa_2)^2N$.  Moreover
$\gcd(\kappa_1,\kappa_2)\sqrt{N}=\epsilon_N^{-1}\sqrt{\kappa_1\epsilon_N}\sqrt{\kappa_2\epsilon_N}$.  This argument can be repeated by replacing $\kappa_1,\kappa_2$ with the other two pairs of $N,\kappa_1$ or $N,\kappa_2$.  That is $N\kappa_1=\gcd(N,\kappa_1)^2\kappa_2$ and $N\kappa_2=\gcd(N,\kappa_2)^2\kappa_1$, and $\gcd(N,\kappa_2)\sqrt{\kappa_1\epsilon_N}=\sqrt{N}\sqrt{\kappa_2\epsilon_N}$ and $\gcd(N,\kappa_1)\sqrt{\kappa_2\epsilon_N}=\sqrt{N}\sqrt{\kappa_1\epsilon_N}$.  Hence two of the generators are dependent if and only if one of $\gcd(\kappa_1,\kappa_2)$, $\gcd(N,\kappa_1)$ and $\gcd(N,\kappa_2)$ is equal to 1.  As $N,\kappa_1,\kappa_2$ are square-free, this happens if and only if one is the product of the other two.
\end{proof}
\end{proof}
\begin{note}\label{ostriknote}
As noted in Section \ref{dnum}, one can upgrade $\mathfrak{D}_N$ to have the structure of a group by allowing $d$-numbers $\alpha$ which are algebraic numbers, as opposed to limiting to algebraic integers.  Let $\mathbb{K}$ be an algebraic number field with $G:=\mathrm{Gal}(\mathbb{K}/\mathbb{Q})$ and unit group $U\subset\mathbb{K}^\times$.  Victor Ostrik has indicated the quotient of $d$-numbers in $\mathbb{K}$ (in the generalized sense) by the rational multiples of units is $H^1(G,U)$, and when $G$ is cyclic of order $n$, $|H^1(G,U)|=n$ when $\mathbb{K}$ contains a unit of norm $-1$, and $|H^1(G,U)|=2n$ otherwise.  This agrees with Lemma \ref{xyz} as any two generators taken from $\{\sqrt{\kappa_1\epsilon_N},\sqrt{\kappa_2\epsilon_N},\sqrt{N}\}$ \emph{rationally} generate the third.  Furthermore, Ostrik has conjectured the stronger statement that when $G$ is cyclic of order $n$, $H^1(G,U)\cong\mathbb{Z}/n\mathbb{Z}$ when $\mathbb{K}$ contains a unit of norm $-1$, and $H^1(G,U)\cong\mathbb{Z}/n\mathbb{Z}\oplus\mathbb{Z}/n\mathbb{Z}$ otherwise. Lemma \ref{xyz} can also be seen as a proof of this statement when $n=2$.
\end{note}
\begin{example}
All five of the possibilities in the conclusion of Theorem \ref{genthm} are realized.  We have $\|\epsilon_2\|=-1$, hence $\mathfrak{D}_2=\langle\sqrt{2}\rangle$, and $\|\epsilon_{15}\|=1$ with $\kappa_1<\kappa_2<N$ and $\kappa_1\kappa_2>N$, hence $\mathfrak{D}_{15}=\langle\sqrt{15},3+\sqrt{15},5+\sqrt{15}\rangle$.  We also have $\|\epsilon_{11}\|=1$ with $\kappa_1=22$, $\kappa_2=2$ and thus $\kappa_2<N<\kappa_1$ with $N\kappa_2=\kappa_1$.  Thus $\mathfrak{D}_{11}=\langle\sqrt{11},3+\sqrt{11}\rangle$.  Note that \begin{equation}
\sqrt{\kappa_2\epsilon_{11}}\sqrt{11}=11+3\sqrt{11}=\sqrt{\kappa_1\epsilon_{11}},
\end{equation}
making the generator $\sqrt{\kappa_1\epsilon_{11}}$ superfluous, as predicted above.  Lastly we have $\|\epsilon_6\|=1$ with $\kappa_1=3$ and $\kappa_2=2$ and thus $\kappa_2<\kappa_1<N$ with $\kappa_1\kappa_2=N$.  Thus $\mathfrak{D}_6=\langle3+\sqrt{6},2+\sqrt{6}\rangle$, with
\begin{equation}
\epsilon_N^{-1}\sqrt{\kappa_1\epsilon_N}\sqrt{\kappa_2\epsilon_N}=(5-2\sqrt{6})(3+\sqrt{6})(2+\sqrt{6})=\sqrt{6}=\sqrt{N}.
\end{equation}
\end{example}
\begin{figure}[H]
\centering
\[
\begin{array}{|c|c|c|c|c|c|c|c|c|c|c|c|c|c|c|c|c|c|c|c|}
\hline N 	& 3  & 6    & 7  & 11 	&14 & 15 	&19	&21	& 22 	& 23	& 30	& 31	& 33	& 34	& 35	& 38&	39&	42&	46	\\\hline
t_N  	  	&  4 &	10 & 16 & 	20 	&30 &  8	&340	&5		&394	&48	&22	&3040	&46	&70	&12	& 74	& 50&	26&	48670\\\hline
\kappa_1 	& 6	 &	3 &  2  & 	22  &2	&	6   	&38	&7		&11	&2		&6		&2		&3		&2		&14	&19&	13&7&	2		\\\hline
\kappa_2 	& 2 &	2 &	 14 &	2	&7	&	10 	&2		&3		&2		&46	&5		&62	&11	&17	&10	&2&	3&6&	23			\\\hline
\end{array}
\]
    \caption{Small square-free $N\in\mathbb{Z}_{\geq2}$ such that $\|\epsilon_N\|=1$, and their associated $\kappa_1$, $\kappa_2$}%
    \label{figkappa}%
\end{figure}
\begin{example}\label{ex1}
Here we list the sets $\mathfrak{D}_N$ for square-free $2\leq N\leq22$, afforded by Theorem \ref{genthm}.
\begin{align*}
&\mathfrak{D}_2=\langle\sqrt{2}\rangle&&\mathfrak{D}_{13}=\langle\sqrt{13}\rangle \\
&\mathfrak{D}_3=\langle\sqrt{3},1+\sqrt{3}\rangle&&\mathfrak{D}_{14}=\langle4+\sqrt{14},7+2\sqrt{14}\rangle \\
&\mathfrak{D}_5=\langle\sqrt{5}\rangle
&&\mathfrak{D}_{15}=\langle\sqrt{15},3+\sqrt{15},5+\sqrt{15}\rangle \\
&\mathfrak{D}_6=\langle3+\sqrt{6},2+\sqrt{6}\rangle
&&\mathfrak{D}_{17}=\langle\sqrt{17}\rangle \\
&\mathfrak{D}_7=\langle\sqrt{7},3+\sqrt{7}\rangle 
&&\mathfrak{D}_{19}=\langle\sqrt{19},13+3\sqrt{19}\rangle\\
&\mathfrak{D}_{10}=\langle\sqrt{10}\rangle 
&&\mathfrak{D}_{21}=\langle(1/2)(7+\sqrt{21}),(1/2)(3+\sqrt{21})\rangle \\
&\mathfrak{D}_{11}=\langle\sqrt{11},3+\sqrt{11}\rangle
&&\mathfrak{D}_{22}=\langle33+7\sqrt{22},14+3\sqrt{22}\rangle
\end{align*}
\end{example}
\begin{corollary}\label{corthree}
Let $N\in\mathbb{Z}_{\geq2}$ be square-free.  If $\alpha\in\mathfrak{D}_N$, there exist $\ell_\alpha,m_\alpha\in\mathbb{Z}$, and $\delta_{\alpha,0},\delta_{\alpha,1},\delta_{\alpha,2}=0,1$ such that
\begin{equation}
\alpha=\ell_\alpha\epsilon_N^{m_\alpha}(\sqrt{N})^{\delta_{\alpha,0}}(\sqrt{\kappa_1\epsilon_N})^{\delta_{\alpha,1}}(\sqrt{\kappa_2\epsilon_N})^{\delta_{\alpha,2}},
\end{equation}
and this factorization is unique if one assumes $\delta_{\alpha,i}=0,1$ and those $\delta_{\alpha,i}$ corresponding to the generators absent in (\ref{twentyfour}) are zero.
\end{corollary}

The notation of Corollary \ref{corthree} will be used consistently and pervasively throughout the remaining sections.  The following lemma allows us to quickly determine possible divisors of quadratic $d$-numbers based on the fact that their ``integral parts'' must divide one another.

\begin{proposition}\label{lemdiv}
Let $N\in\mathbb{Z}_{\geq2}$ be square-free and let $\alpha,\beta\in\mathfrak{D}_N$.  If $\beta\mid\alpha$, then $\ell_\beta\mid\ell_\alpha$.
\end{proposition}

\begin{proof}
Assume there exists an algebraic integer $\gamma$ such that $\alpha=\beta\gamma$ and for each $i=0,1,2$, set $r_i:=1$ if $\delta_{\alpha,i}-\delta_{\beta,i}=-1$ and $r_i:=0$ else.  Then $\gamma(\sqrt{N})^{r_0}(\sqrt{\kappa_1\epsilon_N})^{r_1}(\sqrt{\kappa_2\epsilon_N})^{r_2}\in\mathcal{O}_N$, equal to
\begin{equation}\label{setup}
\dfrac{\ell_\alpha}{\ell_\beta}\epsilon_N^{m_\alpha-m_\beta}(\sqrt{N})^{\delta_{\alpha,0}-\delta_{\beta,0}+r_0}(\sqrt{\kappa_1\epsilon_N})^{\delta_{\alpha,1}-\delta_{\beta,1}+r_1}(\sqrt{\kappa_2\epsilon_N})^{\delta_{\alpha,2}-\delta_{\beta,2}+r_2}
\end{equation}
with $\delta_{\alpha,i}-\delta_{\beta,i}+r_i=0,1$ for $i=0,1,2$ by construction.  Hence $\ell_\beta\mid\ell_\alpha$.
\end{proof}

\begin{corollary}
Let $N\in\mathbb{Z}_{\geq2}$ be square-free and let $\alpha,\beta\in\mathfrak{D}_N$.  If $r_i$ for $i=0,1,2$ are defined as in Proposition \ref{lemdiv} and $\beta\mid\alpha$, then $(\ell_\beta N^{r_0}\kappa_1^{r_1}\kappa_2^{r_2})\mid\ell_\alpha$.
\end{corollary}

\begin{proof}
If $r_0=1$, then $(\sqrt{N})^{-1}=\sqrt{N}/N$ and similarly for $r_1,r_2$.
\end{proof}

\begin{lemma}\label{greaterthenzero}
Let $N\in\mathbb{Z}_{\geq2}$ be square-free.  If $\alpha=a+b\sqrt{N}\in\mathfrak{D}_N$ for $a,b\in\frac{1}{2}\mathbb{Z}_{\geq0}$, then $m_\alpha\geq0$.
\end{lemma}

\begin{proof}
If $a,b\geq0$, then $\sigma(\alpha)\leq\alpha$, which is true if and only if $\sigma(\alpha^2)\leq\alpha^2$.  Define the positive integer $q:=\ell_\alpha^2N^{\delta_{\alpha,0}}\kappa_1^{\delta_{\alpha,1}}\kappa_2^{\delta_{\alpha,2}}$ so that $\alpha^2=q\epsilon_N^{2m_\alpha}$.  Now $\sigma(\alpha^2)\leq\alpha^2$ if and only if $\epsilon_N^{2m_\alpha}\geq\sigma(\epsilon_N^{2m_\alpha})=\epsilon_N^{-2m_\alpha}$, which, as $\epsilon_N>1$, can be true if and only if $m_\alpha\geq0$.
\end{proof}

\begin{theorem}\label{dis}
The set of quadratic $d$-numbers of the form $a+b\sqrt{N}$ with $N\in\mathbb{Z}_{\geq1}$ and $a,b\in\frac{1}{2}\mathbb{Z}_{\geq0}$ is a discrete subset of $\mathbb{R}$.
\end{theorem}

\begin{proof}
If $N\in\mathbb{Z}_{\geq1}$ is fixed and $M\in\mathbb{R}$ is arbitrary, we claim the set of $a+b\sqrt{N}\in\mathfrak{D}_N$ with $a,b\in\frac{1}{2}\mathbb{Z}_{\geq0}$ and $a+b\sqrt{N}\leq M$ is \emph{finite}.  This follows from Lemma \ref{greaterthenzero} and the fact $\epsilon_N>1$.  Now if $N\in\mathbb{Z}_{\geq1}$ is arbitrary, by assumption either $a$ or $b$ is zero (or both), or 
\begin{equation}
(1/2)(1+\sqrt{N})\leq(a+b\sqrt{N})\leq M\label{eqforty}
\end{equation}
from which $N\leq(2M-1)^2$.  Thus for any $M\in\mathbb{R}$ the set of all such quadratic $d$-numbers less than $M$ is a finite union of finite sets and is moreover finite.
\end{proof}

\begin{corollary}\label{discrete}
The set of quadratic $d$-numbers $\alpha$ such that $|\sigma(\alpha)|\leq|\alpha|$ is discrete.
\end{corollary}

\begin{note}
The set of totally positive quadratic $d$-numbers is not discrete.  Indeed $\epsilon^{-2}_N$ is totally positive for all square-free $N\in\mathbb{Z}_{\geq2}$ and $\lim_{k\to\infty}(\epsilon^{-2}_N)^k=0$.
\end{note}
\end{subsection}


\begin{subsection}{On solutions to the negative Pell equation}

The negative Pell Equation is $x^2-Ny^2=-1$ for non-square $N\in\mathbb{Z}_{\geq2}$ and it has been a pervasive open problem to determine for which $N$ there exist integer solutions $x,y$.  A broad spectrum of research exists on this problem which we do not attempt to summarize here (see \cite{etienne,kaplan1986pell,redei1935pellsche,trotter,negativeyokoi} for example).  But it is well-known that this problem is equivalent to the fundamental unit of $\mathbb{Q}(\sqrt{N})$ having norm $-1$.  Theorem \ref{genthm} provides a novel characterization of quadratic fields whose fundamental unit has norm $\pm1$ which we present here.
\begin{corollary}\label{cornormone}
Let $N\in\mathbb{Z}_{\geq2}$ be square-free.  Then $\|\epsilon_N\|=1$ if and only if there exists $\alpha\in\mathfrak{D}_N$ such that $\|\alpha\|$ is not a perfect square and $N$ does not divide $\|\alpha\|$.
\end{corollary}
\begin{proof}
Recall from the proof of Theorem \ref{genthm} that if $\|\epsilon_N\|=1$, there exist square-free $\kappa_1,\kappa_2$ such that $\mathbb{Q}(\sqrt{\kappa_1\kappa_2})=\mathbb{Q}(\sqrt{N})$ which implies $N$ cannot divide both $\kappa_1$ and $\kappa_2$.  Let $\kappa$ be the $\kappa_i$ which is not divisible by $N$.  Then $\sqrt{\kappa\epsilon_N}\in\mathfrak{D}_N$ and $\|\sqrt{\kappa\epsilon_N}\|=\kappa$ which is square-free and not divisible by $N$.  Conversely when $\|\epsilon_N\|=-1$, all $\alpha\in\mathfrak{D}_N$ have a norm which is either a perfect square integer or a perfect square integer multiple of $\|\sqrt{N}\|=-N$ or $\|\epsilon_N\sqrt{N}\|=N$.
\end{proof}
This condition can easily be restated without reference to quadratic $d$-numbers.
\begin{example}\label{kappaequalstwo}
We can find many $d$-numbers satisfying the characterizing condition of Corollary \ref{cornormone} in the case $\|\epsilon_N\|=1$.  It must have minimal polynomial of the form $x^2-nx+\kappa$ for some $n\in\mathbb{Z}_{\geq1}$ where the signs are irrelevant, but we have made a choice so that the resulting $\alpha$ will be positive and maximal among its Galois conjugates.   In the case $\kappa=2$, every \emph{even} $n\in\mathbb{Z}_{\geq1}$ gives the minimal polynomial of a $d$-number $\alpha$ with $\|\alpha\|$ square-free and not divisible by $N$ unless $N=2$.  In particular $\alpha=(1/2)(n+\sqrt{n^2-8})$, so all fields of the form $\mathbb{Q}(\sqrt{N})=\mathbb{Q}(\sqrt{n^2-8})$ for even $n\in\mathbb{Z}_{\geq4}$ and $N\in\mathbb{Z}_{\geq3}$ square-free have $\|\epsilon_N\|=1$.
\end{example}
We now generalize Example \ref{kappaequalstwo}.  If $\kappa\in\mathbb{Z}_{\geq3}$ is square-free, and if a $d$-number $\alpha$ exists of norm $\kappa$ it must have minimal polynomial $x^2-n\kappa x+\kappa$ for some $n\in\mathbb{Z}_{\geq1}$ (where again the choice of signs is irrelevant).  The solution $\alpha=(1/2)(n\kappa+\sqrt{(\kappa n)^2-4\kappa)})$ is a $d$-number with $\|\alpha\|=\kappa$ lying in the field $\mathbb{Q}(\sqrt{(\kappa n)^2-4\kappa})$.  As $\kappa\geq3$ is square-free, then $\kappa n^2-4$ is not divisible by $\kappa$ and hence $\alpha$ is always irrational.  Furthermore, observe $\gcd(\kappa,\kappa n^2-4)$ is 1 or 2 depending if $\kappa$ is odd or even, respectively.  So if $\mathbb{Q}(\sqrt{N})=\mathbb{Q}(\sqrt{(\kappa n)^2-4\kappa})$ for some square-free $N\in\mathbb{Z}_{\geq2}$ then $\kappa$ is divisible by $N$ if and only if $\kappa n^2-4$ is a perfect square, proving the following.
\begin{proposition}\label{pell0}
Let $N\in\mathbb{Z}_{\geq2}$ be square-free.  There exists square-free $\kappa\in\mathbb{Z}_{\geq2}$ and $n\in\mathbb{Z}_{\geq1}$ such that $\kappa n^2-4$ is not a perfect square and $\mathbb{Q}(\sqrt{N})=\mathbb{Q}(\sqrt{(\kappa n)^2-4\kappa})$ if and only if $\|\epsilon_N\|=1$.
\end{proposition}
\begin{corollary}\label{pell}
Let $N\in\mathbb{Z}_{\geq2}$ be square-free.  There exist integer solutions to the negative Pell equation if and only if $\mathbb{Q}(\sqrt{N})\neq\mathbb{Q}(\sqrt{(\kappa n)^2-4\kappa})$ for all square-free $\kappa\in\mathbb{Z}_{\geq2}$ and $n\in\mathbb{Z}_{\geq1}$ such that $\kappa n^2-4$ is not a perfect square.
\end{corollary}
\end{subsection}


\begin{subsection}{Toward a study of weakly quadratic fusion categories}\label{sec:toward}
Recall (Definition \ref{deffour}) that a weakly quadratic fusion category $\mathcal{C}$ is one such that $\mathrm{FPdim}(\mathcal{C})\in\mathfrak{D}_N$ for a square-free $N\in\mathbb{Z}_{\geq2}$, and Lemma \ref{totally} states $\mathrm{FPdim}(\mathcal{C})\geq\sigma(\mathrm{FPdim}(\mathcal{C}))\geq1$, motivating the study of the following subsets of $d$-numbers.
\begin{definition}\label{totallydeaf}
Let $N\in\mathbb{Z}_{\geq2}$ be square-free.  We define
\begin{equation}
\mathfrak{D}_N^{+1}:=\{\alpha\in\mathfrak{D}_N:\alpha\geq\sigma(\alpha)\geq1\},
\end{equation}
and the union of $\mathfrak{D}_N^{+1}$ over all square-free $N\in\mathbb{Z}_{\geq2}$ by $\mathfrak{D}^{+1}$.
\end{definition}
The following result is the basis of our entire study, and is a restricted case of Corollary \ref{discrete}.
\begin{corollary}
Let $N\in\mathbb{Z}_{\geq2}$ be square-free.  The set $\mathfrak{D}^{+1}$ is a discrete subset of $\mathbb{R}$.
\end{corollary}
Theorem \ref{genthm} then allows for an effective bound on the size of $\mathfrak{D}^{+1}$.
\begin{proposition}\label{boond}
Let $M\in\mathbb{Z}_{\geq1}$.  The cardinality of $\mathfrak{D}^{+1}\cap\mathbb{R}_{\leq M}$ is less than $8M(M+1)(2M-1)^2$.
\end{proposition}
\begin{proof}
We assume for simplicity that all fundamental units are norm 1 and all integers are square-free to create a generous upper bound.  With these assumptions, each potential $\alpha\in\mathfrak{D}^{+1}_N$ is indexed by a positive integer less than or equal to $M$, a non-negative integer less than or equal to $M$ (as $(3/2)^n>n$ for $n\geq2$ and $(1+\sqrt{5})/2>3/2$ is the smallest fundamental unit), and at most three selections of $0$ or $1$.  As noted in the proof of Theorem \ref{dis}, there are less than or equal to $(2M-1)^2$ potential $N$. 
\end{proof}
\begin{example}\label{threeplus}
We compute $\mathfrak{D}^{+1}\cap\mathbb{R}_{\leq5}$.  If $\alpha\in\mathfrak{D}^{+1}_N$ for some square-free $N\in\mathbb{Z}_{\geq2}$ we must have $N\leq(2\cdot5-1)^2=81$ by (\ref{eqforty}).  There are $49$ square-free integers less than or equal to $81$.  Of these, 9 are prime $N\leq81$ with $N\equiv1\pmod{4}$ and 4 are $2p$ for some prime $p\equiv5\pmod{8}$.  Theorems 11.5.4 and 11.5.6 of \cite{alaca} state that $\|\epsilon_N\|=-1$ in these cases, respectively.  Of the remaining square-free $N\leq81$, $33$ are divisible by a prime $p\equiv3\pmod{4}$ and thus $\|\epsilon_N\|=1$ in these cases \cite[Theorem 11.5.5]{alaca}.  This leaves 3 to scrutinize by hand: $\|\epsilon_2\|=-1$,  $\|\epsilon_{34}\|=1$ and $\|\epsilon_{65}\|=-1$.  A brutish, but finite computation gives
\begin{equation}
\mathfrak{D}^{+1}\cap\mathbb{R}_{\leq5}=\{1,2,3,(1/2)(5+\sqrt{5}),4,3+\sqrt{3},5\}.
\end{equation}
\end{example}
If one studies the cases $\|\epsilon_N\|=\pm1$ individually, more general results can be proven.  For instance let $N\in\mathbb{Z}_{\geq2}$ be square-free.  If $\|\epsilon_N\|=1$, then $\ell\epsilon_N^m\sqrt{N}\not\in\mathfrak{D}_N^{+1}$ for all $\ell,m\in\mathbb{Z}_{\geq0}$.  To see this, as $\|\epsilon_N\|=1$, $\sigma(\epsilon_N)=\epsilon_N^{-1}$ and thus $\sigma(\ell\epsilon_N^m\sqrt{N})=-\ell\epsilon_N^{-m}\sqrt{N}<0$.  The following stronger result is in the same spirit for units of norm $-1$.
\begin{proposition}\label{thmbound}
Let $N\in\mathbb{Z}_{\geq2}$ be square-free with $\|\epsilon_N\|=-1$.  If $\alpha\in\mathfrak{D}^{+1}_N$, then
\begin{equation}
\ell_\alpha\geq\epsilon_N^{m_\alpha}(\sqrt{N})^{-\delta_0}\qquad\text{ and }\qquad\alpha\geq\epsilon_N^{2m_\alpha}.
\end{equation}
\end{proposition}
\begin{proof}
If $\|\epsilon_N\|=-1$ and $\alpha\in\mathfrak{D}^{+1}_N$, then $m_\alpha\equiv\delta_0\pmod{2}$.  We must have $\ell_\alpha\sigma(\epsilon_N^{m_\alpha}(\sqrt{N})^{\delta_0})\geq1$ by assumption.  As $m_\alpha\equiv\delta_0\pmod{2}$, the signs introduced by the action of $\sigma$ are negated.  Hence $\ell_\alpha\geq\epsilon_N^{m_\alpha}(\sqrt{N})^{-\delta_0}$ and the second statement follows immediately.
\end{proof}

\begin{corollary}\label{lol}
Let $N\in\mathbb{Z}_{\geq2}$ be square-free with $\|\epsilon_N\|=-1$ and $M\in\mathbb{R}_{\geq1}$.  If $\alpha\in\mathfrak{D}_N^{+1}\cap\mathbb{R}_{\leq M}$, then $N+2\sqrt{N}\leq4M-1$.
\end{corollary}

\begin{proof}
We must have $\epsilon_N\geq(1/2)(1+\sqrt{N})$ and the result follows from Proposition \ref{thmbound}.
\end{proof}

\begin{example}
We compute $\bigcup_J(\mathfrak{D}_N^{+1}\cap\mathbb{R}_{\leq 50})$ where $J$ is the subset of square-free $N\in\mathbb{Z}_{\geq2}$ such that $\|\epsilon_N\|=-1$.  Corollary \ref{lol} implies $N=2,5,10,13,29$.  The case $N=2$ contributes $\ell\epsilon_2^2$ for $\ell=6,7,8$ and $\ell\epsilon_2\sqrt{2}$ for $\ell=2\leq\ell\leq14$.  The case $N=5$ contributes $\ell\epsilon_5^2$ for $3\leq\ell\leq19$, $\ell\epsilon_5\sqrt{5}$ for $1\leq\ell\leq13$, $\ell\epsilon_5^3\sqrt{5}$ for $2\leq\ell\leq5$,  and $7\epsilon_5^4$.  The case $N=10$ contributes $2\epsilon_{10}\sqrt{10}$.  The case $N=13$ contributes $\ell_1\epsilon_{13}\sqrt{13}$ for $\ell=1,2,3,4$.  The case $N=29$ contributes one candidate: $\epsilon_{29}\sqrt{29}$.  Moreover there are $57$ such irrational $d$-numbers in addition to the $50$ integers.
\end{example}

For a generic square-free $N\in\mathbb{Z}_{\geq2}$ Tomito and Yamamuro \cite{tomita2002} give lower bounds for $\epsilon_N$ based on the length of the continued fraction expansion of $\sqrt{N}$.    For particular square-free $N\in\mathbb{Z}_{\geq2}$, the bounds given by Proposition \ref{thmbound} are debilitating as we illustrated with the following examples.

\begin{example}\label{largeunit}
Let $N=2593$ (which is prime) \cite[Section 4 (II.b)]{yamamoto1971}.  The fundamental unit is
\begin{equation*}
\epsilon_N=229\,004\,858\,046\,909\,225\,648\,456+4\,497\,212\,789\,358\,213\,431\,953\sqrt{2593}
\end{equation*}
and $\|\epsilon_N\|=-1$.  If $\alpha\in\mathfrak{D}_N^{+1}$, Proposition \ref{thmbound} implies $\alpha\geq\epsilon^2_N>10^{48}$.  Let $N=1054721$ (which is prime).  The fundamental unit is
\begin{equation*}
\epsilon_N=653\,902\,179\,520\,607\,163\,438\,825\,746\,432+636\,713\,397\,684\,223\,825\,329\,255\,425\sqrt{1\,054\,721}
\end{equation*}
and $\|\epsilon_N\|=-1$.  If $\alpha\in\mathfrak{D}_N^{+1}$, Proposition \ref{thmbound} implies $\alpha\geq\epsilon^2_N>10^{60}$.  One can also refer to \cite{reiter1985} for infinite families of square-free $N\in\mathbb{Z}_{\geq2}$ such that $\epsilon_N$ is large (compared to $N$).
\end{example}
\end{subsection}
\end{section}


\begin{section}{Weakly quadratic fusion categories}\label{sec:app}

Here we initiate a study of weakly quadratic fusion categories (Definition \ref{deffour}) using the classification of $d$-numbers given in Theorem \ref{genthm} and Corollary \ref{discrete}.

\begin{corollary}
The set of all $\alpha\in\mathbb{R}$ such that $[\mathbb{Q}(\alpha):\mathbb{Q}]\leq2$ and $\alpha$ is the Frobenius-Perron dimension (or global dimension) of a fusion category is a discrete subset of $\mathbb{R}$.
\end{corollary}

This has not been mentioned explicitly in the literature because it follows from a more general fact.  Corollary 3.13 of \cite{paul} states that for a fixed $M\in\mathbb{R}_{\geq1}$ there are finitely-many fusion categories $\mathcal{C}$ with $\mathrm{FPdim}(\mathcal{C})\leq M$, hence finitely-many possible $d$-numbers $\alpha\leq M$ which are realized as the Frobenius-Perron dimension of a fusion category.  Our major contribution is that for weakly-quadratic fusion categories, we have \emph{constructed} all such possible numbers.


\begin{subsection}{Families of weakly quadratic fusion categories}\label{fam}

A majority of known weakly quadratic fusion categories are weakly integral.  Strictly quadratic fusion categories are encountered less often.  Aside from producing new examples via basic constructions such as products and Drinfeld centers, only finitely-many strictly quadratic fusion categories have been proven to exist.  There are several conjecturally-infinite families of strictly quadratic fusion categories which have been put forward, with finitely-many proofs of existence.  This section is largely expository aside from Proposition \ref{prop1} and its corollary, placing known examples into the framework of Theorem \ref{genthm}.
\label{sec:fam}
\begin{example}\label{quantum}
The largest collection of fusion categories which do not arise from the representation theory of finite groups are semisimple quotients of the representation categories of quantum groups at roots of unity.  To each rank $n\in\mathbb{Z}_{\geq1}$ complex finite-dimensional simple Lie algebra of Dynkin type $X_n$, and \emph{level} $k\in\mathbb{Z}_{\geq1}$ one associates a modular tensor category which we abbreviate $X_{n,k}$.  One can refer to \cite{schopierayprimer} for a technical outline of these examples and further references.   Weakly quadratic categories $X_{n,k}$ are entirely described and the only fields appearing are $\mathbb{Q}(\sqrt{N})$ for  $N=2,3,5,6,21$.  We include all strictly quadratic $X_{n,k}$ here with the Frobenius-Perron dimensions factored as in Corollary \ref{corthree}.
\begin{figure}[H]
\centering
\begin{equation*}
\begin{array}{|c|c|c|c|}
\hline X & n & k & \mathrm{FPdim}(X_{n,k})\\\hline
A 	& 1 & 3 & 2\epsilon_5\sqrt{5} \\
 	& 1 & 6 & 8\epsilon_2\sqrt{2}  \\
 	& 1 & 8 & 20\epsilon_5^2 \\
 	& 1 & 10 & 24\epsilon_3 \\[0.25cm]
 	& 2 & 2 & 3\epsilon_5\sqrt{5} \\
 	& 2 & 5 & 42\epsilon_2^2 \\
 	& 2 & 7 & 60\epsilon_5^5\sqrt{5} \\
 	& 2 & 9 & 432\epsilon_3^2 \\[0.25cm]
 	& 3 & 4 & 128\epsilon_2^2 \\
 	& 3 & 6 & 800\epsilon_5^6 \\
 	& 3 & 8 & 3456\epsilon_3^3\\[0.25cm]\hline
\end{array}
\,\,
\begin{array}{|c|c|c|c|}
\hline X & n & k & \mathrm{FPdim}(X_{n,k})\\\hline
A   & 4 & 3 & 80\epsilon_2^2 \\
 	& 4 & 5 & 2000\epsilon_2^6 \\
 	& 4 & 7 & 8640\epsilon_3^4 \\[0.25cm]
 	& 5 & 2 & 24\epsilon_2\sqrt{2} \\
 	& 5 & 4 & 1200\epsilon_5^6 \\
 	& 5 & 6 & 20736\epsilon_3^4 \\[0.25cm]
 	& 6 & 3 & 140\epsilon_5^5\sqrt{5} \\ 	
 	& 6 & 5 &  12096\epsilon_3^4 \\[0.25cm] 	
 	& 7 & 2 & 80\epsilon_5^2 \\ 	
 	& 7 & 4 &  6912\epsilon_3^3 \\
 	& & & \\\hline
\end{array}
\,\,
\begin{array}{|c|c|c|c|}
\hline X & n & k & \mathrm{FPdim}(X_{n,k})\\\hline
A 	& 8 & 3 & 1296\epsilon_3^2 \\[0.25cm]
 	& 9 & 2 & 120\epsilon_3 \\[0.25cm]
E 	& 6 & 3 & 45\epsilon_5^3\sqrt{5} \\[0.25cm]
 	& 7 & 2 & 3\epsilon_5\sqrt{5} \\
 	& 7 & 3 & 42\epsilon_{21} \\[0.25cm] 
F 	& 4 & 1 &  \epsilon_5\sqrt{5} \\
	& 4 & 3 & 48\epsilon_6 \\[0.25cm]
G 	& 2 & 1 & \epsilon_5\sqrt{5} \\
	& 2 & 3 & 21\epsilon_{21} \\
	& & & \\\hline	
\end{array}
\end{equation*}
    \caption{Strictly quadratic $X_{n,k}$}%
    \label{figlie}%
\end{figure}
\end{example}
The following conjecturally-infinite family of examples originated in the work of Masaki Izumi but appeared first in print by Siehler in \cite{siehler2003near}.  Later they were studied (with their Drinfeld centers) through the lens of operator algebras by Evans and Gannon in \cite{evans2014near}.  A very detailed and thorough exposition on this approach was given by Izumi in \cite{izumi2017cuntz}.
\begin{definition}
A \emph{near group fusion category} is a fusion category with exactly one noninvertible simple object, up to isomorphism.
\end{definition}
One can easily index such categories.  The set of invertible simple objects forms a finite group $G$.  Let $\rho$ be any noninvertible simple object.  For some $k\in\mathbb{Z}_{\geq0}$ we then have
\begin{equation}\label{sixty}
\rho^2=\sum_{g\in G}g+k\rho
\end{equation}
in the Grothendieck ring.   As such we say a near-group fusion category is of type $(G,k)$, generically denoted $\mathcal{N}(G,k)$.

\begin{example}\label{tambara}
The categories $\mathcal{N}(G,0)$ are referred to as a \emph{Tambara-Yamagami} fusion categories and one immediately sees these categories are weakly integral, with $\mathrm{FPdim}(\mathcal{N}(G,0))=2|G|$.
\end{example} 

It was shown in \cite[Theorem A.6]{ost15} that when $\mathrm{FPdim}(\rho)$ is irrational, then $k$ is a multiple of $|G|$, and so henceforth $k=n|G|$ for some $n\in\mathbb{Z}_{\geq1}$.  From (\ref{sixty}), $\mathrm{FPdim}(\rho)$ is the largest root of the quadratic polynomial $x^2-n|G|x-|G|$ which is a $d$-number by Lemma \ref{def} (v) and one verifies
\begin{equation}
\mathrm{FPdim}(\mathcal{N}(G,n|G|))^2=(n^2|G|^2+4|G|)\mathrm{FPdim}(\rho)^2.
\end{equation}
As a $d$-number, $\mathrm{FPdim}(\rho)^2$ is an integer multiple of a unit and as $\|\mathrm{FPdim}(\rho)^2\|=|G|^2$, then necessarily $\mathrm{FPdim}(\rho)^2=|G|u$ for some unit $u$.  In particular,
\begin{equation}
\dfrac{1}{|G|}\mathrm{FPdim}(\rho)^2=\dfrac{1}{2}\left(n^2|G|+2+n\sqrt{n^2|G|^2+4|G|}\right)
\end{equation}
is a unit.  If the existence of near-group categories is prolific, they represent examples of strictly quadratic fusion categories over a broad set of fields.  A conjecturally-infinite family of strictly quadratic fusion categories which are not near-group arose as a generalization of the fusion categories constructed via the Haagerup subfactor \cite{asaeda1999exotic}.  As with the hypothetical near-group examples these, often called Haagerup-Izumi categories, have been studied via operator algebras \cite{haagerupizumi}.  The reader can reference \cite{izumi2016classification} for details and finite-many explicit constructions.

\begin{definition}
Let $G$ be a finite group.  A fusion category of \emph{Haagerup-Izumi} type $G$, denoted $\mathcal{HI}_G$, is a rank $2|G|$ fusion category with $|G|$ invertible objects indexed by elements of $G$ and their fusion given by group multiplication, and $|G|$ non-invertible simple objects $\{g\rho:g\in G\}$ whose fusion is given in the Grothendieck ring by
\begin{equation}
g(h\rho)=(gh)\rho=(h\rho)g^{-1},\text{ and }(g\rho)(h\rho)=gh^{-1}+\sum_{a\in G}a\rho.
\end{equation}
\end{definition}
The noninvertible objects have the same Frobenius-Perron dimension subject to
\begin{equation}
\mathrm{FPdim}(\rho)^2-|G|\mathrm{FPdim}(\rho)-1=0.
\end{equation}
Note $\mathrm{FPdim}(\rho)=(1/2)(|G|+\sqrt{|G|^2+4})$ is a unit of norm $-1$ in $\mathbb{Q}(\sqrt{|G|^2+4})$.  Thus we have
\begin{equation}
\mathrm{FPdim}(\mathcal{HI}_G)=|G|(1+\mathrm{FPdim}(\rho)^2)=|G|\mathrm{FPdim}(\rho)\sqrt{|G|^2+4}.
\end{equation}
The following example shows that in general, $\mathcal{HI}_G$ is not the unique fusion category $\mathcal{C}$ with $\mathrm{FPdim}(\mathcal{C})=\mathrm{FPdim}(\mathcal{HI}_G)$.
\begin{example}
Let $|G|=11$ so that $\mathrm{FPdim}(\mathcal{HI}_G)=11\epsilon_5\sqrt{11^2+4}=55\epsilon_5\sqrt{5}$.  Thus
\begin{equation}
\mathrm{FPdim}(G_{2,1}\boxtimes\mathrm{Vec}_{\mathbb{Z}/5\mathbb{Z}}\boxtimes\mathrm{Vec}_{\mathbb{Z}/11\mathbb{Z}})=55\epsilon_5\sqrt{5}.
\end{equation}
These fusion categories are not equivalent because they differ in rank.
\end{example}
\begin{definition}
A fusion category $\mathcal{C}$ is called a \emph{generalized near-group fusion category} if $\mathcal{O}(\mathcal{C}_\mathrm{pt})$ acts (by left $\otimes$) transitively on $\mathcal{O}(\mathcal{C})\backslash\mathcal{O}(\mathcal{C}_\mathrm{pt})$.
\end{definition}
When there exists precisely one isomorphism class of non-invertible objects, generalized near-group categories devolve into near-group categories.  Let $G$ be the set of isomorphism classes of invertible objects in a generalized near-group fusion category $\mathcal{C}$ and $G_\rho$ be the stabilizer of $\rho\in\mathcal{O}(\mathcal{C})$.  We will abuse notation by referring to cosets in $G/G_\rho$ by a representative element in said class.  With these conventions, it was shown in \cite[Proposition IV.2.2]{thornton2012generalized} that each noninvertible $\rho\in\mathcal{O}(\mathcal{C})$ satisfies
\begin{equation}
\rho^2=\sum_{g\in G_\rho}g+\sum_{h\in G/G_\rho}k_h(h\rho)
\end{equation}
in the Grothendieck ring of $\mathcal{C}$ where $k_h\in\mathbb{Z}_{\geq0}$ for all $h\in G/G_\rho$.  Furthermore, all noninvertible $\rho\in\mathcal{O}(\mathcal{C})$ have identical decompositions, from which the generalized near-group category $\mathcal{C}$ is determined by the collection $(G,G_\rho,\{k_h:h\in G/G_\rho)$ for any $\rho\in\mathcal{O}(\mathcal{C})$.  This implies
\begin{equation}
\mathrm{FPdim}(\rho)^2-K\mathrm{FPdim}(\rho)-|G_\rho|=0
\end{equation}
where $K:=\sum_{h\in G/G_\rho}k_h$.  One may then compute
\begin{equation}\label{eqfooorty}
\mathrm{FPdim}(\mathcal{C})=[G:G_\rho](\mathrm{FPdim}(\rho)^2+|G_\rho|)=[G:G_\rho]\mathrm{FPdim}(\rho)\sqrt{K^2+4|G_\rho|}.
\end{equation}
\begin{proposition}\label{prop1}
If $\mathcal{C}$ is a generalized near-group category of type $(G,G_\rho,\{k_h:h\in G/G_\rho)$, then $\mathrm{FPdim}(\rho)$ is a $d$-number.
\end{proposition}
\begin{proof}
By Equation (\ref{eqfooorty}), $\mathrm{FPdim}(\mathcal{C})^2/\mathrm{FPdim}(\rho)^2=[G:G_\rho]^2(K^2+4|G_\rho|)\in\mathbb{Z}$.  The result then follows from Lemma \ref{dfp} and Corollaries 2.8 and 2.9 of \cite{codegrees}.
\end{proof}
\begin{corollary}\label{corgen}
If $\mathcal{C}$ is a generalized near-group category of type $(G,G_\rho,\{k_h:h\in G/G_\rho)$ and $K:=\sum_{h\in G/G_\rho}k_h$, then $|G_\rho|$ divides $K^2$.
\end{corollary}
\begin{proof}
Proposition \ref{prop1} shows $\mathrm{FPdim}(\rho)=(1/2)(K+\sqrt{K^2+4|G_\rho|})$ is a $d$-numer.  Thus
\begin{equation}
\dfrac{4(K/2)^2}{\|\mathrm{FPdim}(\rho)\|}=\dfrac{K^2}{-|G_\rho|}\in\mathbb{Z}
\end{equation}
by Lemma \ref{lemma} (iv).
\end{proof}
\end{subsection}


\begin{subsection}{Non-existence results}\label{sec:non}

It is well-known that each positive integer $n$ is realized as the Frobenius-Perron dimension of the fusion category $\mathrm{Vec}_G$ for any finite group $G$ with $|G|=n$.  For quadratic extensions, $\mathfrak{D}^{+1}$ is a discrete subset of $\mathbb{R}$, but unlike the rational case there exist elements of $\mathfrak{D}^{+1}$ which cannot be realized as the Frobenius-Perron dimension of fusion categories.  The smallest such example is $3+\sqrt{3}$ from Example \ref{threeplus}.  Let $X$ be a simple object in a fusion category $\mathcal{C}$ such that $\mathrm{FPdim}(\mathcal{C})=3+\sqrt{3}$.  Then
\begin{equation}
\mathrm{FPdim}(X)\leq\sqrt{\mathrm{FPdim}(\mathcal{C})-1}=\sqrt{2+\sqrt{3}}=2\cos(\pi/12)<2.
\end{equation}
It is known by the work of Kronecker \cite[Corollary 3.3.16]{tcat} that if $\mathrm{FPdim}(X)<2$, then $\mathrm{FPdim}(X)=2\cos(\pi/n)$ for some $n\geq3$ and thus in our case, $\mathrm{FPdim}(X)=2\cos(\pi/n)$ for $n=3,\ldots,12$.  A finite computation then ensures no such fusion category $\mathcal{C}$ exists.
\begin{note}
Although no fusion category exists of Frobenius-Perron dimension $3+\sqrt{3}$, there exists a fusion category $\mathcal{C}$ with $\mathrm{FPdim}(\mathcal{C})=4(3+\sqrt{3})$.  One has $\mathrm{FPdim}(A_{1,10})=24\epsilon_3$ (Figure \ref{figlie}) and there exists a commutative algebra $A\in A_{1,10}$ of Frobenius-Perron dimension $3+\sqrt{3}$, whose category of modules has Frobenius-Perron dimension $24\epsilon_3/(3+\sqrt{3})=4(3+\sqrt{3})$ \cite[Lemma 3.11]{DMNO}.
\end{note}
Clearly this structure of argument cannot be extended to produce an infinite list of elements of $\mathfrak{D}^{+1}$ which cannot be realized as the Frobenius-Perron dimension of a fusion category.  In this section we prove a key result, Proposition \ref{selfadjoint}, which aids in the proof that no pseudounitary fusion category exists of the form $p\epsilon_N$ where $p\in\mathbb{Z}_{\geq2}$ is prime and $N\in\mathbb{Z}_{\geq2}$ is square-free (Proposition \ref{doesnotexist}).  The following definition is a specialized case of the notion of a \emph{quantum integer}.

\begin{definition}
Let $N\in\mathbb{Z}_{\geq2}$ be square-free.  For $m\in\mathbb{Z}$ we define $[m]_{N}:=(\epsilon_N^m-\epsilon_N^{-m})/(\epsilon_N-\epsilon_N^{-1})$.
\end{definition}
\begin{lemma}\label{thmlemma}
Let $N\in\mathbb{Z}_{\geq2}$ be square-free.  For any $m\in\mathbb{Z}$, $[m]_{N}\in\mathcal{O}_N$.  Furthermore, if $\|\epsilon_N\|=1$ or $\|\epsilon_N\|=-1$ and $m$ is odd, $[m]_N\in\mathbb{Z}$ and if $\|\epsilon_N\|=-1$ and $m$ is even, $[m]\in\mathbb{Z}\cdot\sqrt{N}$.
\end{lemma}
\begin{proof}
The first statement follows from the long division
\begin{equation}
\dfrac{\epsilon_N^m-\epsilon_N^{-m}}{\epsilon_N-\epsilon_N^{-1}}=\epsilon_N^{-m-1}+\epsilon_N^{-m+1}+\cdots+\epsilon_N^{m-3}+\epsilon_N^{m-1}
\end{equation}
and the fact that $\mathcal{O}_N$ is a ring.  Now compute
\begin{equation}
\sigma([m]_N)=\dfrac{(\|\epsilon_N\|\epsilon_N)^{-m}-(\|\epsilon_N\|\epsilon_N)^m}{(\|\epsilon_N\|\epsilon)^{-1}_N-(\|\epsilon_N\|\epsilon_N)}=\dfrac{-\|\epsilon_N\|^m(\epsilon_N^{m}-\epsilon_N^{-m})}{-\|\epsilon_N\|(\epsilon_N-\epsilon_N^{-1})}=\|\epsilon_N\|^{m-1}[m]_N.
\end{equation}
Hence if $\|\epsilon_N\|=1$, $[m]_N$ is an algebraic integer in $\mathbb{Q}$, hence $[m]_N\in\mathbb{Z}$.  Furthermore if $\|\epsilon_N\|=-1$, then $[m]_N\in\mathbb{Z}$ when $m$ is odd, and $[m]_N\in\mathbb{Z}\cdot\sqrt{N}$ when $m$ is even.
\end{proof}

Recall the set of simple objects $X\in\mathcal{C}$ such that $\mathrm{FPdim}(X)^2\in\mathbb{Z}$ generates a fusion subcategory $\mathcal{C}_\mathrm{int}\subset\mathcal{C}$ \cite[Lemma 3.5.6]{tcat}.

\begin{proposition}\label{selfadjoint}
Let $N\in\mathbb{Z}_{\geq2}$ be square-free and $\mathcal{C}$ be a fusion category such that $\mathrm{FPdim}(\mathcal{C})=\ell\epsilon_N^m$ for some $\ell,m\in\mathbb{Z}_{\geq1}$.  If $\mathrm{FPdim}(X)\in\mathfrak{D}_N$ for all simple $X\in\mathcal{C}$, there exists a sequence of nonnegative integers $\{\ell_j\}_{j=1}^\infty$ (finitely-many nonzero) such that
\begin{equation}\label{thmeq}
[m]_N\mathrm{FPdim}(\mathcal{C}_\mathrm{int})=\sum_{j=1}^\infty\ell_j[j-m]_N.
\end{equation}
\end{proposition}

\begin{proof}
As $\mathrm{FPdim}(X)\in\mathfrak{D}_N$ for all simple $X\in\mathcal{C}$,  Theorem \ref{genthm} implies
\begin{equation}
\mathrm{FPdim}(X)^2\in\{\ell_X\epsilon_N^{m_X}:\ell_X\in\mathbb{Z}_{\geq1}\text{ and }m_X\in\mathbb{Z}_{\geq0}\}.
\end{equation}
So we may decompose
\begin{equation}
\mathrm{FPdim}(\mathcal{C})=\ell\epsilon_N^m=\sum_{j=1}^\infty\ell_j\epsilon_N^j+\mathrm{FPdim}(\mathcal{C}_\mathrm{int})
\end{equation}
where $\ell_j=\sum_{m_X=j}\ell_X$.  Multiplying both sides by $\epsilon_N^{-m}$ yields
\begin{equation}
\ell=\sum_{j=1}^\infty\ell_j\epsilon_N^{j-m}+\mathrm{FPdim}(\mathcal{C}_\mathrm{int})\epsilon_N^{-m}
\end{equation}
and so the right-hand side is invariant under $\sigma\in\mathrm{Gal}(\mathbb{Q}(\sqrt{N})/\mathbb{Q})$.  Hence
\begin{align}
&&\sum_{j=1}^\infty\ell_j\epsilon_N^{j-m}+\mathrm{FPdim}(\mathcal{C}_\mathrm{int})\epsilon_N^{-m}&=\sum_{j=1}^\infty\ell_j(\|\epsilon_N\|\epsilon_N)^{m-j}+\mathrm{FPdim}(\mathcal{C}_\mathrm{int})(\|\epsilon_N\|\epsilon_N)^m \\
\Rightarrow&&\sum_{j=1}^\infty\ell_j\dfrac{\epsilon_N^{j-m}-(\|\epsilon_N\|\epsilon_N)^{m-j}}{(\|\epsilon_N\|\epsilon_N)^m-\epsilon_N^{-m}}&=\mathrm{FPdim}(\mathcal{C}_\mathrm{int})
\end{align}
which is well-defined as $m\neq0$.  If $\|\epsilon_N\|=-1$, then $m$ is even because $\mathrm{FPdim}(\mathcal{C})$ is totally positive, and for all $\ell_j\neq0$, $j$ is even as well, or else $\ell_\alpha\epsilon_N^j$ is not totally positive, even though it is a sum of totally real squares.  So whether $\|\epsilon_N\|=\pm1$ our result is proven.
\end{proof}

To see the breadth of Proposition \ref{selfadjoint}, recall the \emph{adjoint subcategory} $\mathcal{C}_\mathrm{ad}$ of a fusion category $\mathcal{C}$ \cite[Definition 4.14.5]{tcat}, generated by $X\otimes X^\ast$ for all $X\in\mathcal{O}(\mathcal{C})$.  In general, $\mathrm{FPdim}(X)^2\in\mathbb{Q}(\mathrm{FPdim}(\mathcal{C}))$ for all $X\in\mathcal{O}(\mathcal{C})$, hence if $\mathrm{FPdim}(\mathcal{C})\in\mathfrak{D}_N$, then $\mathrm{FPdim}(X)\in\mathbb{Q}(\sqrt{N})$ for all $X\in\mathcal{O}(\mathcal{C}_\mathrm{ad})$.  And moreover if $\mathcal{C}$ is any pseudounitary fusion category $\mathcal{C}$, $\mathcal{Z}(\mathcal{C})_\mathrm{ad}$ automatically satisfies the hypotheses of Proposition \ref{selfadjoint}.

\begin{note}
We emphasize that (\ref{thmeq}) is an equality of \emph{integers}.  Certainly this is true by Lemma \ref{thmlemma} if $\|\epsilon_N\|=1$, but the final comments of the proof of Proposition \ref{selfadjoint} imply the left and right-hand sides of Equation (\ref{thmeq}) are multiples of $\sqrt{N}$ if $\|\epsilon_N\|=-1$.  Dividing by $\sqrt{N}$ yields an equality of integers.
\end{note}

\begin{example}\label{geee2}
Consider a fusion category $\mathcal{C}$ satisfying the hypotheses of Proposition \ref{selfadjoint} with Frobenius-Perron dimension $21\epsilon_{21}$.  We claim there are exactly two lists of Frobenius-Perron dimensions of simple objects such that $\mathrm{FPdim}(\mathcal{C})=21\epsilon_{21}$ (hence in each case the rank is determined).  Proposition \ref{selfadjoint} implies that there exists a sequence of nonnegative integers $\{\ell_j\}_{j=1}^\infty$ such that $\mathrm{FPdim}(\mathcal{C}_\mathrm{int})=\sum_{j=1}^\infty\ell_j[j-1]_{21}$.  We have $[j-1]_{21}=0,1,5,24,\ldots$ for $j=1,2,3,4\ldots$, so as $\mathrm{FPdim}(\mathcal{C}_\mathrm{int})$ divides $21$, we need only find potential $\ell_2,\ell_3$ and $\ell_1$ will be determined from them.  But $\epsilon_{21}^3>21\epsilon_{21}$ and $5\epsilon_{21}^2>21\epsilon_{21}$, hence $\ell_3=0$, $\ell_2\leq4$, and $\ell_1\leq21$.  Therefore, there are $336$ potential equalities of the form $\mathrm{FPdim}(\mathcal{C}_\mathrm{int})+\ell_1\epsilon_{21}+\ell_2\epsilon_{21}^2=21\epsilon_{21}$ to check.  Two solutions are possible.  One solution is $\mathrm{FPdim}(\mathcal{C}_\mathrm{int})=3$, $\ell_1=6$ and $\ell_2=3$.  In this case $\mathcal{C}_\mathrm{int}$ is pointed of rank 3.  There must be $3$ objects with Frobenius-Perron dimension $\epsilon_{21}$ (as $\sqrt{2\epsilon_{21}}$ and $\sqrt{3\epsilon_{21}}$ are not $d$-numbers), and $2$ objects of dimension $\sqrt{3\epsilon_{21}}$ (as $\sqrt{\ell\epsilon_{21}}$ for $\ell=1,2,4,5,6$).  The second solution has trivial $\mathcal{C}_\mathrm{int}$.  By the same reasoning as in the first case, there must be 3 objects of dimension $\sqrt{3\epsilon_{21}}$, one object of dimension $\epsilon_{21}$, and one object of dimension $\sqrt{7\epsilon_{21}}$.  The second case is realized by the quantum group category $G_{2,3}$ (see Example \ref{quantum}) while the first is realized by a category Morita equivalent to $G_{2,3}$.
\end{example}

\begin{note}Example \ref{geee2} illustrates an unusually strong constraint on weakly quadratic fusion categories $\mathcal{C}$ with Frobenius-Perron dimension $\ell\epsilon_N$ when for all $X\in\mathcal{O}(\mathcal{C})$, $\mathrm{FPdim}(X)^2$ is an integer multiple of $1$, $\epsilon_N$ or $\epsilon_N^2$.  Proposition \ref{selfadjoint} implies there exist $\ell_1,\ell_2$ such that $\mathrm{FPdim}(\mathcal{C}_\mathrm{int})+\ell_1\epsilon_N+\ell_2\epsilon_N^2=\ell\epsilon_N$.  Hence
\begin{equation}\label{fiftythree}
\mathrm{FPdim}(\mathcal{C}_\mathrm{int})+(\ell_1-\ell)\epsilon_N+\ell_2\epsilon_N^2=0.
\end{equation}
As the minimal polynomial of $\epsilon_N$ is $x^2-t_Nx+1$, Equation \ref{fiftythree} is determined entirely by $\mathrm{FPdim}(\mathcal{C}_\mathrm{int})$, which implies $\ell_2=\mathrm{FPdim}(\mathcal{C}_\mathrm{int})$ and $\ell_1=\ell-t_N\mathrm{FPdim}(\mathcal{C}_\mathrm{int})$.  This argument by minimal polynomial of $\epsilon_N$ is vital in the proof of the lemmas leading to Proposition \ref{doesnotexist}.
\end{note}

The remainder of this exposition will use Proposition \ref{selfadjoint} to prove Proposition \ref{doesnotexist} which states prime integer multiples of fundamental units don't arise as Frobenius-Perron dimensions of pseudounitary fusion categories.  Our technique is inductive.  In particular Theorem \ref{genthm} implies that one can inductively study quadratic $d$-numbers via the exponent of the fundamental unit.  It is not clear if all weakly quadratic fusion categories are Galois conjugate to a pseudounitary category, though this has been proven for generalized near-group categories \cite{thornton2012generalized}.  As such we may not use the fact that $\mathrm{FPdim}(\mathcal{C})$ is divisible by $\mathrm{FPdim}(X)^2$ for all $X\in\mathcal{O}(\mathcal{C})$ without assuming pseudounitarity.  Furthermore, a braiding is required to conclude that $\mathrm{FPdim}(X)$ is a $d$-number for all simple $X\in\mathcal{O}(\mathcal{C})$ when $\mathcal{C}$ is pseudounitary.  Therefore to prove our results for generic fusion categories we must pass through the Drinfeld center $\mathcal{Z}(\mathcal{C})$.  Lastly, we frequently use two basic constructions of fusion categories: factorization \cite[Theorem 4.2]{mug1} and de-equivariantization \cite[Theorem 8.23.3]{tcat} which in our case, on the level of Frobenius-Perron dimensions, correspond to factorization and division of $d$-numbers.

\begin{lemma}\label{previouszero}
Let $N,\kappa\in\mathbb{Z}_{\geq2}$ be square-free.  If $\mathcal{C}$ is a weakly quadratic braided fusion category with $\mathrm{FPdim}(\mathcal{C})=\sqrt{\kappa\epsilon_N}$, then $\mathcal{C}$ is modular and has no proper nontrivial fusion subcategories.
\end{lemma}

\begin{proof}
The symmetric center of $\mathcal{C}$ is necessarily integral (as it is Tannakian/super-Tannakian by Deligne's Theorem), hence $\|\mathrm{FPdim}(\mathcal{C}')\|$ is a perfect square.  Because $\|\mathrm{FPdim}(\mathcal{C})\|=\kappa$ is square-free, then $\mathrm{FPdim}(\mathcal{C}')$ cannot divide $\mathrm{FPdim}(\mathcal{C})$ unless $\mathcal{C}'$ is trivial and thus $\mathcal{C}$ is modular.  Similarly, \emph{any} nontrivial proper fusion subcategory $\mathcal{D}\subset\mathcal{C}$ must have $\mathrm{FPdim}(\mathcal{D})\in\mathfrak{D}^{+1}_N$ with square-free norm.  But by Theorem \ref{genthm}, up to units and products, there are at most 3 elements of $\mathfrak{D}_N$ with square-free norm: $\sqrt{N}$, $\sqrt{\kappa_1\epsilon_N}$, and $\sqrt{\kappa_2\epsilon_N}$.  As $\sqrt{N}\not\in\mathfrak{D}^{+1}_N$, we must have $\mathrm{FPdim}(\mathcal{D})=\epsilon_N^m\sqrt{\ell\epsilon_N}$ for some  $\ell\in\mathbb{Z}_{\geq1}$ and $m\in\mathbb{Z}_{\geq0}$.  But by the same reasoning as before, $\mathcal{D}$ is modular, and $\mathcal{C}\simeq\mathcal{D}\boxtimes\mathcal{E}$ for some fusion subcategory $\mathcal{E}\subset\mathcal{C}$ with $\mathrm{FPdim}(\mathcal{E})=\epsilon_N^{-m}\sqrt{\kappa\epsilon_N}/\sqrt{\ell\epsilon_N}=\epsilon_N^{-m}\sqrt{\kappa/\ell}$.  But from $\mathrm{FPdim}(\mathcal{E})\in\mathfrak{D}^{+1}_N$ we may conclude $m=0$, and $\kappa/\ell=1$ or $\kappa/\ell=N$.  The latter cannot happen because $\sqrt{N}\not\in\mathfrak{D}_N^{+1}$ and the former implies $\mathcal{E}$ is trivial.  Thus $\mathcal{C}$ has no proper nontrivial fusion subcategories.
\end{proof}

\begin{lemma}\label{previous}
Let $N\in\mathbb{Z}_{\geq2}$ be square-free, and $p\in\mathbb{Z}_{\geq2}$ prime.  No pseudounitary modular tensor category $\mathcal{C}$ exists with $\mathrm{FPdim}(\mathcal{C})=p\epsilon_N$.
\end{lemma}

\begin{proof}
By Lemma \ref{subcategory}, $\mathrm{FPdim}(\mathcal{C})/\mathrm{FPdim}(\mathcal{C}_\mathrm{int})$ is an algebraic integer.  Proposition \ref{lemdiv} then implies $\mathrm{FPdim}(\mathcal{C}_\mathrm{int})=1,p$.  If $\mathrm{FPdim}(\mathcal{C}_\mathrm{int})=p$, then by \cite[Corollary 8.30]{ENO}, $\mathcal{C}_\mathrm{int}$ is pointed of rank $p$.  The symmetric center $\mathcal{C}_\mathrm{int}'$ is either trivial or $\mathcal{C}$ is Tannakian.  In either case one may construct a fusion category of Frobenius-Perron dimension $\epsilon_N$ (either by factoring \cite[Theorem 4.2]{mug1} or de-equivariantization \cite[Example 3.8]{DMNO}), which does not exist because $\epsilon_N\not\in\mathfrak{D}_N^{+1}$.  So we must have $\mathrm{FPdim}(\mathcal{C}_\mathrm{int})=1$.  Furthermore if $\mathcal{D}\subset\mathcal{C}$ is any proper nontrivial fusion subcategory, $\|\mathrm{FPdim}(\mathcal{D})\|=1,p,p^2$.  By Theorem \ref{genthm}, if $\|\mathrm{FPdim}(\mathcal{D})\|=p^2$, then $\mathrm{FPdim}(\mathcal{D})=p\epsilon_N^m$ for some $m\in\mathbb{Z}_{\geq0}$.  But $m\neq0$ as $\mathcal{D}$ is not weakly integral, and $m$ is not a positive integer as $\mathcal{D}$ is proper.  So $
\|\mathrm{FPdim}(\mathcal{D})\|=p$ which as in the proof of Lemma \ref{previouszero} implies $\mathrm{FPdim}(\mathcal{D})=\sqrt{p\epsilon_N}$.  Both $\mathcal{D}$ and $\mathcal{E}$ have no nontrivial proper fusion subcategories, so Lemma \ref{previouszero} then implies $\mathcal{D}$ is modular and $\mathcal{C}\simeq\mathcal{D}\boxtimes\mathcal{E}$ where $\mathcal{E}\subset\mathcal{C}$ is another fusion subcategory with $\mathrm{FPdim}(\mathcal{E})=\sqrt{p\epsilon_N}$.  Moreover $\mathcal{C}_\mathrm{ad}\simeq\mathcal{D}_\mathrm{ad}\boxtimes\mathcal{E}_\mathrm{ad}=\mathcal{D}\boxtimes\mathcal{E}$ (see the proof of Proposition 2.2 \cite{DMNO}) and $\mathcal{C}=\mathcal{C}_\mathrm{ad}$.  So we have proven $\mathrm{FPdim}(X)\in\mathfrak{D}_N$ for all $X\in\mathcal{O}(\mathcal{C})$.  Proposition \ref{selfadjoint} implies the existence of a sequence of nonnegative integers $\{\ell_j\}_{j=1}^\infty$ such that $1=\sum_{j=1}^\infty\ell_j[j-1]_N$.  Hence there exists exactly one object of squared Frobenius-Perron dimension $\epsilon_N^2$ as $\|\epsilon_N\|=1$ and $[n]_N\geq2$ for $n\geq2$.  Moreover $p\epsilon_N=1+\ell_1\epsilon_N+\epsilon_N^2$, or $\epsilon_N^2+(\ell_1-p)\epsilon_N+1=0$.  Pseudounitarity and Lemma \ref{modular} imply $\mathrm{FPdim}(\mathcal{C})/\mathrm{FPdim}(X)^2$ is an algebraic integer for all simple $X\in\mathcal{C}$.  Hence Proposition \ref{lemdiv} states that aside from the two simple objects of unique dimension, all other simple objects must have squared Frobenius-Perron dimension $\epsilon_N$ or $p\epsilon_N$.  But $\ell-p=-t_N$ as $\epsilon_N$ is the fundamental unit, so $\mathrm{FPdim}(X)^2=p\epsilon_N$ is impossible.  Thus $\mathrm{FPdim}(X)=\sqrt{\epsilon_N}$ for any nontrivial simple $X$ whose Frobenius-Perron dimension is not $\epsilon_N$.  This is a contradiction to $\mathcal{C}$ being self-adjoint.  Therefore the only possibility is there is exactly one invertible and one non-invertible object.  The category $\mathcal{C}$ cannot be rank 2 though as these categories have Frobenius-Perron dimension $2$ or $\epsilon_5\sqrt{5}$ \cite{ostrik}.
\end{proof}

\begin{corollary}
Let $N,\kappa\in\mathbb{Z}_{\geq2}$ be square-free.  No weakly quadratic pseudounitary fusion category $\mathcal{C}$ exists with $\mathrm{FPdim}(\mathcal{C})=\sqrt{\kappa\epsilon_N}$.
\end{corollary}

\begin{proof}
The Drinfeld center of $\mathcal{C}$ is pseudounitary and modular with $\mathrm{FPdim}(\mathcal{Z}(\mathcal{C}))=p\epsilon_N$.
\end{proof}

\begin{lemma}\label{previous2}
Let $N\in\mathbb{Z}_{\geq2}$ be square-free, and $p\in\mathbb{Z}_{\geq2}$ prime.  If $\mathcal{C}$ is a pseudounitary modular tensor category with $\mathrm{FPdim}(\mathcal{C})=p\epsilon^2_N$ then $\|\epsilon_N\|=-1$.
\end{lemma}

\begin{proof}
Akin to the proof of Lemma \ref{previous}, $\mathcal{C}_\mathrm{int}$ must be trivial and $\mathcal{C}=\mathcal{C}_\mathrm{ad}$ as these facts only relied on $\|\mathrm{FPdim}(\mathcal{C})\|=p^2$ and $\epsilon_N^m\not\in\mathfrak{D}^{+1}_N$ for all $m\in\mathbb{Z}_{\geq1}$.  Hence the hypotheses of Proposition \ref{selfadjoint} are satisfied once again.  Proposition \ref{selfadjoint} allows the existence of a sequence of nonnegative integers $\{\ell_j\}_{j=1}^\infty$ such that $[2]_N=\sum_{j=1}^\infty\ell_j[j-2]_N$.  But $[-1]_N=[1]_N=1$, so if there exists $X\in\mathcal{O}(\mathcal{C})$ with $\mathrm{FPdim}(X)^2=\epsilon_N^4$ (there can be no larger as $[n]_N>[2]_N$ for $n\geq3$), then $X$ is unique up to isomorphism, and
\begin{equation}
p\epsilon_N^2=1+\ell_2\epsilon_N^2+\epsilon_N^4,
\end{equation}
and thus $\epsilon_N^2+\ell_2-p+\epsilon_N^{-2}=0$. But $\epsilon_N^{-1}=t_N-\epsilon_N$, hence $\epsilon_N$ has minimal polynomial
\begin{equation}
\epsilon_N^2+\ell_2-p+(t_N-\epsilon_N)^2=2\epsilon_N^2-2t_N\epsilon_N+\ell_2-p+t_N^2=0
\end{equation}
Now if $\mathrm{FPdim}(X)^2=\ell_X\epsilon_N^2$ for some $\ell_X\in\mathbb{Z}_{\geq1}$, then $\ell_X=1,p$, the latter being impossible as then $\mathrm{FPdim}(X)^2=\mathrm{FPdim}(\mathcal{C})$.  Thus $\mathrm{FPdim}(X)=\epsilon_N$ for each $X\in\mathcal{O}(\mathcal{C})$ except the unit object and one distinguished object of Frobenius-Perron dimension $\epsilon_N^2$.  But this implies for all $X\in\mathcal{O}(\mathcal{C})$ with $\mathrm{FPdim}(X)=\epsilon_N$,  $X\otimes X^\ast\cong\mathbbm{1}\oplus Y$ where $Y$ is a sum of $n$ simple objects of dimension $\epsilon_N$.  Hence
\begin{equation}
\epsilon_N^2-n\epsilon_N-1=0,
\end{equation}
and $\|\epsilon_N\|=-1$.  If there does not exist $X\in\mathcal{O}(\mathcal{C})$ with $\mathrm{FPdim}(X)^2=\epsilon_N^4$ then we have
\begin{equation}\label{inconsistent}
0=1+\ell_1\epsilon_N+(\ell_2-p)\epsilon_N^2+\ell_3\epsilon_N^3.
\end{equation}
The existence of $X\in\mathcal{O}(\mathcal{C})$ with $\mathrm{FPdim}(X)^2=\ell_x\epsilon_N^3$ for some $\ell_x\in\mathbb{Z}_{\geq1}$ implies $\ell_x=1,p$ by Lemma \ref{modular}.  The case  $\ell_x=p$ implies $\mathrm{FPdim}(X)^2>\mathrm{FPdim}(\mathcal{C})$ and the case $\ell_x=1$ implies $\mathrm{FPdim}(X)=\epsilon_N\sqrt{\epsilon_N}\not\in\mathfrak{D}_N$.  Therefore $\ell_3=0$ and $\ell_1=[2]_N=t_N$ as $\|\epsilon_N\|=1$ (or else $[2]_N\not\in\mathbb{Z}$).  But this is inconsistent, as Equation \ref{inconsistent} then becomes
\begin{equation}\label{fiftynine}
(p-\ell_2)\epsilon_N^2-t_N\epsilon_N-1=0.
\end{equation}
Hence $p-\ell_2=\pm1$, but in either case, Equation (\ref{fiftynine}) is not the minimal polynomial of $\epsilon_N$.
\end{proof}

\begin{lemma}\label{previous3}
Let $N\in\mathbb{Z}_{\geq2}$ be square-free, and $p\in\mathbb{Z}_{\geq2}$ prime.  If $\mathcal{C}$ is a pseudounitary modular tensor category with $\mathrm{FPdim}(\mathcal{C})=p^2\epsilon^2_N$, then $\|\epsilon_N\|=-1$.
\end{lemma}

\begin{proof}
As in the previous two proofs we note $\mathrm{FPdim}(\mathcal{C}_\mathrm{int})=1,p,p^2$.  If $\mathrm{FPdim}(\mathcal{C}_\mathrm{int})=p^2$, then by \cite[Proposition 8.32]{ENO} either $p$ is odd and $\mathcal{C}_\mathrm{int}$ is pointed, or $p=2$ and $\mathcal{C}_\mathrm{int}$ is a a $\mathbb{Z}/2\mathbb{Z}$ Tambara-Yamagami category (see Example \ref{tambara}).  In the latter case we require $\sigma(4\epsilon_N^2)\geq1$ which implies $2\geq\epsilon_N$.  This is only true for $N=5$ and $\|\epsilon_5\|=-1$.  Now assume $\mathcal{C}_\mathrm{int}$ is pointed of rank $p^2$.  There must exist a pointed fusion subcategory $\mathcal{D}\subset\mathcal{C}_\mathrm{int}$ of rank $p$ which is either nondegenerate or Tannakian.  In the latter case one can de-equivariantize $\mathcal{C}$ to obtain a fusion category of Frobenius-Perron dimension $p\epsilon_N^2$ \cite[Lemma 3.11]{DMNO} and then $\|\epsilon_N\|=-1$ by Lemma \ref{previous2}.  And in the former case one can factor $\mathcal{C}_\mathrm{int}$ into two nondegenerate pointed factors of rank $p$, hence $\mathcal{C}$ factors into $\mathcal{C}_\mathrm{int}$ and an impossible fusion category with Frobenius-Perron dimension $\epsilon_N^2\not\in\mathfrak{D}_N^{+1}$.  Now we consider the case $\mathrm{FPdim}(\mathcal{C}_\mathrm{int})=p$ and thus $\mathcal{C}_\mathrm{int}$ is pointed of rank $p$.  For the same reasons as above, this implies $\|\epsilon_N\|=-1$ by Lemma \ref{previous2}.
\end{proof}

\begin{example}
Lemma \ref{previous2} and Lemma \ref{previous3} cannot be extended to all square-free $N\in\mathbb{Z}_{\geq2}$ with $\|\epsilon_N\|=-1$ as $\mathrm{FPdim}(\mathcal{Z}(G_{2,1}))=5\epsilon_5^2$ and $\mathrm{FPdim}(\mathrm{Vec}_{\mathbb{Z}/5\mathbb{Z}}\boxtimes\mathcal{Z}(G_{2,1}))=5^2\epsilon_5^2$.
\end{example}

\begin{proposition}\label{doesnotexist}
Let $N\in\mathbb{Z}_{\geq2}$ be square-free and $p\in\mathbb{Z}_{\geq2}$ prime.  No pseudounitary fusion category $\mathcal{C}$ exists with $\mathrm{FPdim}(\mathcal{C})=p\epsilon_N$.
\end{proposition}

\begin{proof}
As $p\epsilon_N\in\mathfrak{D}_N^{+1}$, then $\|\epsilon_N\|=1$.  The Drinfeld center of $\mathcal{C}$ is a pseudounitary modular tensor category with $\mathrm{FPdim}(\mathcal{Z}(\mathcal{C}))=p^2\epsilon_N^2\in\mathfrak{D}^{+1}_N$, contrary to Lemma \ref{previous3}.
\end{proof}

\end{subsection}


\end{section}

\begin{section}{Further discussion}
Many topics covered in this exposition deserve further investigation.  We separate the following into concerns related to number theory and those regarding fusion categories.

\begin{subsection}{Number theory}
Recall the following conjecture from Note \ref{ostriknote} and the comments following Definition \ref{deforder}.
\begin{conjecture}[Ostrik]\label{ost}
Let $\mathbb{K}$ be a cyclic extension of $\mathbb{Q}$ of degree $n\in\mathbb{Z}_{\geq2}$.  If $G$ is the Galois group of $\mathbb{K}$ over $\mathbb{Q}$ and $U$ is the unit group of $\mathbb{K}$, then
\begin{equation}
\mathfrak{G}_\mathbb{K}=H^1(G,U)=\left\{\begin{array}{lcc}\mathbb{Z}/n\mathbb{Z} & : & \exists u\in U\text{ with }\|u\|=-1 \\ \mathbb{Z}/n\mathbb{Z}\times \mathbb{Z}/n\mathbb{Z} & : & \text{else}\end{array}\right..
\end{equation}
\end{conjecture}
Theorem \ref{genthm} shows that even in the case $n=2$, there exist fields $\mathbb{K}$ for which the rational generators of $\mathfrak{G}_\mathbb{K}$ do not generate $\mathfrak{D}_\mathbb{K}$ integrally.
\begin{example}
Consider the cubic extension $\mathbb{K}=\mathbb{Q}(2\cos(\pi/7))$.  By the Dirichlet unit theorem, $\mathcal{O}_\mathbb{K}^\times=\{\pm\epsilon_1^{m_1}\epsilon_2^{m_2}:m_1,m_2\in\mathbb{Z}\}$ where $\epsilon_1=2\cos(\pi/7)$ and $\epsilon_2=\epsilon_1^2-2=2\sin(3\pi/14)$.  As $\|\epsilon_1\|=-1$, Conjecture \ref{ost} states that $\mathfrak{G}_\mathbb{K}$ is generated by a single $d$-number.  We find that $\alpha:=\sqrt[3]{7\epsilon_1^2\epsilon_2}$ is such a generator with minimal polynomial $x^3-7x-7$.  Is $\mathfrak{D}_\mathbb{K}$ generated integrally by $\alpha$?
\end{example}
\begin{question}
If $\mathbb{K}$ is a cyclic extension of $\mathbb{Q}$, what conditions ensure the minimal generating set of $\mathfrak{G}_\mathbb{K}$ (as a group) is also a minimal generating set for $\mathfrak{D}_\mathbb{K}$ (as a monoid)?
\end{question}
\begin{conjecture}\label{contwo}
Let $\mathbb{K}$ be a totally real algebraic number field with $[\mathbb{K}:\mathbb{Q}]>2$.  Then
\begin{equation}
\{\alpha\in\mathfrak{D}_\mathbb{K}:\alpha\geq\sigma(\alpha)\geq1\text{ for all }\sigma\in\mathrm{Gal}(\mathbb{K}/\mathbb{Q})\}
\end{equation}
is not a discrete subset of $\mathbb{R}$.
\end{conjecture}
Even if Conjecture \ref{contwo} is true, classifying $d$-numbers in higher-degree extensions is an interesting number-theoretic question.  A similar classification to Theorem \ref{genthm} could possibly be attained for totally real biquadratic extensions of $\mathbb{Q}$ using the work of Kubota \cite{kubota}.  In particular \cite[Satz 1]{kubota} states all generators of the unit group of $\mathbb{Q}(\sqrt{M},\sqrt{N})$ for some square-free $M,N\in\mathbb{Z}_{\geq2}$ can be expressed in terms of the quadratic fundamental units $\epsilon_M$, $\epsilon_N$ and $\epsilon_{MN}$ and their square roots.
\end{subsection}


\begin{subsection}{Fusion categories}\label{confusion}
Example \ref{largeunit} provides evidence that perhaps there exist quadratic number fields which do not contain any ``new'' Frobenius-Perron dimensions of fusion categories (not integers), as the smallest such fusion category would be unusually large relative to the discriminant of the field.
\begin{question}\label{quest1}
Does there exist a square-free $N\in\mathbb{Z}_{\geq2}$ such that if $\mathcal{C}$ is a fusion category with $\mathrm{FPdim}(\mathcal{C})\in\mathbb{Q}(\sqrt{N})$, then $\mathcal{C}$ is weakly integral?
\end{question}
So few general constructions of weakly quadratic fusion categories exist that we dare not conjecture an answer to Question \ref{quest1} at this time, but it seems reasonable to conjecture that \emph{if} the answer is yes, then there exist infinitely-many such fields.
\begin{question}
Can all fusion categories with $\mathrm{FPdim}(\mathcal{C})\in\mathfrak{D}_2$ be constructed from known examples?
\end{question}
\begin{question}
Are all weakly quadratic fusion categories Galois conjugate to a pseudounitary fusion category?
\end{question}
The Frobenius-Perron dimensions of simple objects in a pseudounitary weakly quadratic fusion category need not be $d$-numbers as is shown by example.  But the Frobenius-Perron dimensions of simple objects in the Drinfeld center of such a category \emph{are} $d$-numbers.  This suggests there may be a way to classify Frobenius-Perron dimensions which are not $d$-numbers by passing through the Drinfeld center.
\begin{question}
Is there a description of Frobenius-Perron dimensions of simple objects in a weakly quadratic fusion category in terms of $d$-numbers?  Is this possible for an arbitrary fusion category?
\end{question}
Based on the similarity between the analogous result for near-group categories and Corollary \ref{corgen} we also make the following conjecture.
\begin{conjecture}
If $\mathcal{C}$ is a generalized near-group category of type $(G,G_\rho,\{k_h:h\in G/G_\rho)$ and $K:=\sum_{h\in G/G_\rho}k_h$, then $|G_\rho|$ divides $K$.
\end{conjecture}
\end{subsection}

\end{section}

\bibliographystyle{plain}
\bibliography{bib}

\end{document}